\documentclass[11pt, english, draft]{article}
\usepackage{amssymb,amsmath}
\title{\bf CERTAIN PROPERTIES OF THE POWER GRAPH ASSOCIATED WITH A FINITE GROUP\thanks{
Supported by the National Natural Science Foundation of China
(Grant No. 11171364)}}
\author{{\bf A. R. Moghaddamfar} \ and \ {\bf S. Rahbariyan} \\[0.1cm]
{\em Department of Mathematics, K. N. Toosi
University of Technology,}\\
 {\em P. O. Box $16315$-$1618$, Tehran, Iran}\\[0.1cm]
{\em E-mails}:  {\tt moghadam@kntu.ac.ir} {\em and} {\tt
moghadam@mail.ipm.ir}\\[0.3cm]
{\bf W. J. Shi}\\[0.1cm] {\em Department of Mathematics and Statistics, Chongqing University
of Arts and Sciences,}\\ {\em Chongqing $402160$, China}\\ {\em
E-mail:} {\tt wjshi@suda.edu.cn}}

\newenvironment{proof}{\noindent {\em {Proof}}.}{$\square$
\medskip}

\newtheorem{corollary}{Corollary}

\newtheorem{theorem}{Theorem}
\newtheorem{proposition}{Proposition}
\newtheorem{lm}{Lemma}

\newtheorem{qu}{Question}

\begin{document}
\maketitle
\begin{abstract}
\noindent There are a variety of ways to associate directed or
undirected graphs to a group. It may be interesting to
investigate the relations between the structure of these graphs
and characterizing certain properties of the group in terms of
some properties of the associated graph. The power graph
$\mathcal{P}(G)$ of a group $G$ is a simple graph whose
vertex-set is $G$ and two vertices $x$ and $y$ in $G$ are
adjacent if and only if $y=x^m$ or $x=y^m$ for some positive
integer $m$. We also pay attention to the subgraph
$\mathcal{P}^\ast(G)$ of $\mathcal{P}(G)$ which is obtained by
deleting the vertex 1 (the identity element of $G$). In the
present paper, we first investigate some properties of the power
graph $\mathcal{P}(G)$ and the subgraph $\mathcal{P}^\ast(G)$. We
next prove that many of finite groups such as finite simple
groups, symmetric groups and the automorphism groups of sporadic
simple groups can be uniquely determined by their power graphs
among all finite groups. We have also determined up to
isomorphism the structure of any finite group $G$ such that the
graph $\mathcal{P}^\ast(G)$ is a strongly regular graph, a
bipartite graph, a planar graph or an Eulerian graph. Finally, we
obtained some infinite families of finite groups such that the
graph $\mathcal{P}^\ast(G)$ containing some cut-edges.
\end{abstract} {\small {\it $2000$ Mathematics
Subject Classification}: 20D05, 20D06, 20D08, 05C10, 05C45.
\\[0.1cm]
}

\renewcommand{\baselinestretch}{1}
\def\thefootnote{ \ }
\section{Introduction}
Given an algebraic structure $S$, there are different ways to
associate a directed or undirected graph to $S$ in such a way the
vertices are associated with families of elements or subsets of
$S$ and in which two vertices are joined by an arc or by an edge
if and only if they satisfy a certain relation. Since a graph
(directed or undirected) can be investigated in terms of the
results from Graph Theory, one can obtain some information about
the structure of $S$. In other words, we are interested in
characterizing certain properties of $S$ in terms of some
properties of the associated graph. This has been a fruitful
topic in the last years.

{\em Notation and Definitions.} We begin by introducing some
well-known graphs associated with semigroups or groups.

$\bullet$ {\em The Power graph.} Let $S$ be a semigroup and $X$ a
non-empty subset of $S$. The {\em directed power graph} on $X$,
denoted by $\vec{\mathcal{P}}(S, X)$, has $X$ as its vertex-set
and for two distinct vertices $x, y\in X$ there is an arc from
$x$ to $y$ if and only if $y=x^m$ for some positive integer $m$.
Similarly, the {\em (undirected) power graph} $\mathcal{P}(S,
X)$, where $S$ is a semigroup and $X$ a non-empty subset of $S$,
is the graph with vertex-set $X$ and such that two vertices $x,
y\in X$ are joined by an edge (and we write $x\sim y$) if $x\neq
y$ and $y=x^m$ or $x=y^m$ for some positive integer $m$. In the
case when $X=S$, we denote the directed (resp. undirected) power
graphs by $\vec{\mathcal{P}}(S)$ (resp. $\mathcal{P}(S)$). The
graphs $\vec{\mathcal{P}}(S)$ and $\mathcal{P}(S)$ have been
introduced and studied for the first time in \cite{K1} and
\cite{CSS}, respectively. Given a group $G$, we denote by
$\mathcal{P}^\ast(G)=\mathcal{P}(G, G\setminus\{1\})$ (resp.
$\vec{ \mathcal{P}}^\ast(G)=\vec{\mathcal{P}}(G,
G\setminus\{1\})$), the graph (resp. direct graph) obtained by
deleting the vertex $1$ (identity element) from $G$.

Generally, one can study the power graph $\mathcal{P}(S, X)$ for
different choices of $S$ and $X$, and from a number of different
perspectives. In particular, some other algebraic structures such
as: groups, rings, fields etc., are  interesting choices for $S$.
In recent years, the directed and undirected power graphs have
been investigated by many authors. For earlier results concerning
these graphs we refer to \cite{C1, C2, CSS, K1, K2, K3}. In
particular, in \cite{C1}, the power graph $\mathcal{P}(G)$ is
studied in the case when $G$ is an abelian group.

$\bullet$ {\em The commuting graph and its complement.} Another
graph that has attracted the attention of many authors is the
{\em commuting graph} associated with a finite group. For a finite
group $G$ and $X$ a non-empty subset of $G$, the commuting graph
on $X$ denoted $\mathcal{C}(G, X)$ has $X$ as its vertex-set with
$x, y\in X$ joined by an edge whenever $xy=yx$. Many authors have
studied $\mathcal{C}(G, X)$ for different choices of $G$ and $X$.
For instance, in the case when $X$ is a set of involutions, then
$\mathcal{C}(G, X)$ is called a commuting involution graph. In
particular, commuting involution graphs for arbitrary involution
conjugacy classes of symmetric groups were considered in
\cite{Ba}. Moreover, in the case when $G$ is a finite nonabelian
group and $X=G\setminus Z(G)$, the non-central elements of $G$,
we call $\mathcal{C}(G, X)$ the {\em commuting graph of $G$}. In
\cite{Se} and \cite{SS}, Segev and Seitz apply the commuting
graph $\Delta(G):=\mathcal{C}(G, G\setminus \{1\})$ of $G$, with
$G$ a nonabelian simple group, in order to prove the
Margulis-Platonov conjecture on arithmetic groups.

The {\em noncommuting graph} of a nonabelian group $G$, denoted
$\nabla(G)$, is defined as follows: its vertices are the
non-central elements of $G$, and two vertices are adjacent when
they do not commute. Noncommuting graphs have been investigated by
many authors (for instance, see \cite{AAM, Mszz, Mogh, Ne}).

$\bullet$ {\em The prime graph or Gruenberg-Kegel graph.} Another
graph which has deserved a lot of attention is the {\em prime
graph} or {\em Gruenberg-Kegel graph} $\Gamma(G)$ of a finite
group $G$. In this graph, the vertices are the prime numbers
dividing the order of the group $G$ and two different vertices
$p$ and $q$ are joined when $G$ possesses an element of order
$pq$. As a matter of example, for a finite nilpotent group, this
graph is complete, but for the alternating groups of degree $5$
or $6$, this graph has three isolated vertices. The first
references of the prime graph known to Gruenberg and Kegel in an
unpublished manuscript, and Williams, Kondratev, Iiyori and
Yamaki, who studied the number of connected components of the
prime graph of finite groups (see \cite{Gru, IY1, IY2, K, W1,
W2}).

{\em Graph notation.} Two vertices which are incident with a
common edge are adjacent, and two distinct adjacent vertices are
neighbours. The set of neighbours of a vertex $v$ in a graph
$\Gamma$ is denoted by $N_\Gamma(v)$. The degree (or valency) of
a vertex $v$ in a graph $\Gamma$, denoted by $d_\Gamma(v)$, is
the number of neighbours of $v$.  When the graph $\Gamma$ is clear
from the context we simply denote $N(v)$ and $d(v)$. A graph
$\Gamma$ is $k$-regular if $d(v)=k$ for all $v\in V$. A spanning
subgraph of a graph $\Gamma$ is a subgraph obtained by edge
deletions only, in other words, a subgraph whose vertex-set is
the entire vertex set of $\Gamma$. Recall that a null graph
(empty graph) is a graph without edges. A star graph consists of
one central vertex having edges to other vertices in it. An
independent set is a set of vertices in a graph, no two of which
are adjacent. A clique in a graph is a set of vertices all
pairwise adjacent. Two graphs are disjoint if they have no vertex
in common. The union of simple graphs $\Gamma_1$ and $\Gamma_2$ is
the graph $\Gamma_1\cup\Gamma_2$ with vertex set $V(\Gamma_1)\cup
V(\Gamma_2)$ and edge set $E(\Gamma_1)\cup E(\Gamma_2)$. If
$\Gamma_1$ and $\Gamma_2$ are disjoint, we refer to their union
as a disjoint union, and generally denote it by $\Gamma_1\oplus
\Gamma_2$. By starting with a disjoint union of two graphs
$\Gamma_1$ and $\Gamma_2$ and adding edges joining every vertex
of $\Gamma_1$ to every vertex of $\Gamma_2$, one obtains the join
of $\Gamma_1$ and $\Gamma_2$, denoted $\Gamma_1\vee \Gamma_2$. A
digraph is an ordered pair of sets $D=(V, A)$, where $V$ is the
set of vertices and $A$ is the set of ordered pairs (called arcs)
of vertices of $V$. Given a digraph $D=(V, A)$ and $x\in V$ we
will denote by $N_D^+(x)$ (resp. $N_D^-(x)$) the set of the
out-neighbours (resp. in-neighbours) of $x$, i.e., $\{y\in V \ |
\ (x, y)\in A\}$ (resp. $\{y\in V \ | \ (y, x)\in A\}$). The
out-degree $d_D^+(v)$ (resp. in-degree $d_D^-(v)$) of $v$ is the
number of in-neighbours (resp. out-neighbours) of $v$. In other
words, $d_D^\epsilon(v)=|N_D^\epsilon(v)|$, where
$\epsilon\in\{+, -\}$. A vertex of out-degree zero is called a
{\em sink}, one of in-degree zero a {\em source}. Note that, as
with graphs, we will drop the letter $D$ from our notation
whenever possible, thus we will write $N^+(v)$, $N^-(v)$,
$d^+(v)$ and $d^-(v)$ for $N_D^+(v)$, $N_D^-(v)$, $d_D^+(v)$ and
$d_D^-(v)$, respectively.

{\em Group Notation.} We denote by $\pi(n)$ the set of the prime
divisors of a positive integer $n$. Given a group $G$, we shall
write $\pi(G)$ instead of $\pi(|G|)$. We denote the order of an
element $x$ in $G$ by $o(x)$. A group $G$ is periodic if every
element in $G$ has finite order. The spectrum of a periodic group
$G$ is a subset $\pi_e(G)$ of the set of natural numbers
consisting of all orders of elements of $G$. This set is closed
and partially ordered by divisibility; hence, it is uniquely
determined by $\mu(G)$, a subset of its elements that are maximal
under the divisibility relation. A group is called an
$EPPO$-group if every element of the group has prime power order,
i.e., $\pi_e(G)=\{p^s, q^t, \ldots\}$, where $p^s$, $q^t$,
$\ldots$ are primes powers (see \cite{shi-yang-86}). Especially,
a group in which every non-trivial element has prime order, i.e.,
$\pi_e(G)=\{1, p, q, \ldots \}$, where $p$, $q$, $\ldots$ are
primes, is also called an $EPO$-group (see  \cite{shi-wenze}).
Moreover, a group $G$ is called an $EPMO$-group if $\pi_e(G)= \{1,
p, q, \ldots\, r, m\}$, where $p$, $q$, $\ldots$, $r$ are primes
(see \cite{Shi-Yang}), and a group $G$ is called an $EPPMO$-group
if $\pi_e(G)=\{p^s, q^t, \ldots\, r^k, m\}$, where $p^s$, $q^t$,
$\ldots$, $r^k$ are prime powers (see \cite{Huang-Shi}). Denote
by $A_m$ and $S_m$ the alternating and symmetric groups of degree
$m$, respectively. The notation $\mathbb{Z}_m^n$ means that the
direct product if $n$ copies of $\mathbb{Z}_m$ (the cyclic group
of order $m$) and $G=[N]K$ denotes the split extension of a
normal subgroup $N$ of $G$ by a complement $K$.

All further unexplained notation and terminology for graphs,
semigroups, and groups are standard and we refer the reader to
\cite{Ch, Cli, Gor, Gr, Ho, Rob, west}, for instance.

In \cite{C1}, Cameron and Ghosh asked the following question:
\begin{qu}\label{q1}
Is it true two groups with isomorphic power graphs must
themselves be isomorphic?
\end{qu}

They showed this question has an affirmative answer for finite
abelian groups \cite{C1}. In other words, a finite {\em abelian}
group is determined up to isomorphism by its power graph.
However, the answer to the Question \ref{q1} is negative in
general, as the following examples show (see \cite{C1}):

$\bullet$ {\em Infinite abelian groups}.  In fact, the power graph
of the Pr$\rm \ddot{u}$fer group $C_{p^\infty}$ is a countable
complete graph, for each prime $p$.

$\bullet$ {\em Finite groups of order $27$ and exponent $3$.}
Consider the elementary abelian group $P$ of order 27, and the
group of order 27 with the following presentation:
$$G=\langle x, y \ | \ x^3=y^3=[x,y]^3=1\rangle,$$ where $[x, y]$ is
the commutator $x^{-1}y^{-1}xy$. It is easy to see that, these
groups have isomorphic power graphs, while they are
non-isomorphic groups.

In this paper we will continue the study of the power graph
$\mathcal{P}(G)$ and $\mathcal{P}^\ast (G)$ as well. Among the
other results, we classify the finite groups whose power graphs
$\mathcal{P}(G)$ or $\mathcal{P}^\ast (G)$ has a certain graph
property.

{\em A few words about the contents.} The paper consists of seven
sections. In Section 2, we first provide some elementary results
used throughout the paper, and also we present a brief discussion
on the connectivity of the graph $\mathcal{P}^\ast (G)$ and the
relation between power graph and the other graphs. In Section 3,
we show that a finite simple group, a symmetric group and the
automorphism group of a sporadic simple group can be uniquely
determined by their power graphs among all finite groups. In the
rest of paper, we determine up to isomorphism the structure of any
finite group $G$ for which $\mathcal{P}^\ast (G)$ is strongly
regular, bipartite, planar or Eulerian, and we obtained some
infinite families of finite groups such that the graph
$\mathcal{P}^\ast(G)$ containing some cut-edges.
\section{Preliminaries}
In this section we collect some of the results that will be
needed later.
\subsection{Some Elementary Results}
We begin with the following lemma which is taken from \cite{CSS}.
\begin{lm}\label{complete-CSS}{\rm \cite[Theorem 2.12]{CSS}} \
Let $G$ be a finite group. Then  $\mathcal{P}(G)$ is complete if
and only if $G$ is a cyclic group of order $1$ or $p^m$ for some
prime number $p$ and for some natural number $m$.
\end{lm}

We next give some elementary facts concerning power graphs which
follows from the definition.
\begin{lm}\label{elementary-results}
Let $G$ be a finite group. Then the following conditions hold.
\begin{itemize}
\item[$(a)$]If $H$ is a subgroup of $G$, then $\mathcal{P}(H)\leq
\mathcal{P}(G)$. In particular, if $x$ is a $p$-element of $G$,
where $p$ is a prime, then $\langle x\rangle$ is a clique in
$\mathcal{P}(G)$.
\item[$(b)$] If $x\in G$ and $o(x)\in \mu(G)$, then as a vertex
in the power graph $\mathcal{P}(G)$ we have $d(x)=o(x)-1$. (Note
that, in general, if $x\in G$, then $o(x)-1\leq d(x)\leq
|C_G(x)|-1$.) Therefore, the power graph $\mathcal{P}(G)$ of an
elementary $p$-group $G$ of order $p^n$ consist of
$$(p^n-1)/(p-1)=p^{n-1}+p^{n-2}+\cdots+p+1,$$ complete subgraphs on $p$ vertices sharing
a common vertex (the identity element). In particular, when $p=2$
the power graph $\mathcal{P}(G)$ is a star graph. In general, the
power graph $\mathcal{P}(G)$ of an $EPPO$-group $G$ is the union
of complete subgraphs with exactly one common vertex, i.e., the
identity element.

\item[$(c)$] If $x, y\in G$ are elements of coprime orders, or if $x,
y\in G$ are involutions, then $x\nsim y$ in the power graph
$\mathcal{P}(G)$. In particular, the set of involutions of $G$ is
an independent set in $\mathcal{P}(G)$.
\end{itemize}
\end{lm}
{\em Proof.}
\begin{itemize}
\item[$(a)$] The assertion follows immediately from
Proposition 4.5 in \cite{CSS} and Lemma \ref{complete-CSS}.

\item[$(b)$] Evidently $x\sim x^i$ for each $2\leq i\leq o(x)$, hence
$d(x)\geq o(x)-1$. Now, let $y\in G\setminus \langle x\rangle$
and $y\sim x$. By the definition of power graph we have $x\in
\langle y\rangle$ or equivalently $x=y^k$ for some $k$. But then
$o(x)|o(y)$, which implies that $o(y)=o(x)$ because $o(x)\in
\mu(G)$, and so $\langle y\rangle=\langle x\rangle$, an
impossible. Therefore $d(x)=o(x)-1$. The rest is obvious.

\item[$(c)$] Indeed, if $x\sim y$, then by the definition $x\in \langle
y\rangle$ or $y\in \langle x\rangle$, hence $o(y)|o(x)$ or
$o(x)|o(y)$, which is a contradiction. If  $x, y\in G$ are
distinct involutions, then $x\not\in \langle y\rangle$ and $y\not
\in \langle x\rangle$, hence $x\nsim y$.
\end{itemize}
The proof is complete. $\Box$

Similarly, the following simple lemma is a direct consequence of
the definition.
\begin{lm}\label{elementary-results-digraphs}
Let $G$ be a finite group and $D=\vec{\mathcal{P}}(G)$. Then the
following conditions hold.
\begin{itemize}
\item[$(a)$]If $H$ is a subgroup of $G$, then $\vec{\mathcal{P}}(H)\leq
D$.
\item[$(b)$] If $x\in G$ and $o(x)\in \mu(G)$, then as a vertex
in the direct power graph $D$ we have $$N_D^+(x)=\{1, x^2, x^3,
\ldots, x^{o(x)-1}\} \ \ \mbox{and} \ \ \ \ N_D^-(x)=\{x^n \ | \
(n, o(x))=1, n\neq 1\}.$$ In particular, we have $d_D^+(x)=o(x)-1$
and $d_D^-(x)=\phi(o(x))-1$.
\item[$(c)$] If $x\in G$ is an involution, then $d_D^+(x)=1$. In other words, every involution of $G$ is
a sink in $\vec{\mathcal{P}}^\ast(G)$.
\item[$(d)$] $D$ does not contain any source.
\item[$(e)$] If $x^n$ and $x^m$ are two generators of $\langle x\rangle$,
then $$N_D^+(x^n)\setminus\{x^m\}=N_D^+(x^m)\setminus\{x^n\} \ \
\mbox{and} \ \
N_D^-(x^n)\setminus\{x^m\}=N_D^-(x^m)\setminus\{x^n\}.$$
\end{itemize}
\end{lm}
\begin{proof}
These are straightforward verifications.
\end{proof}

Next two results are elementary but important for further
investigations.

\begin{lm}\label{FACT} Let $a$ and $b$ be two elements of $G$ such that the order of
one of them does not divide the order of the other one and they
commute, then $a$ and $b$ lie in the same component of
$\mathcal{P}^\ast(G)$.
\end{lm}
\begin{proof} Without loss of generality we may assume that
$o(a)<o(b)$ and $o(b)\equiv n \pmod{o(a)}$. Then, since
$$a^{n}=(ab)^{o(b)}\in \langle ab\rangle \ \ \ \ {\rm and} \ \ \ \ b^{o(a)}=(ab)^{o(a)}\in \langle ab\rangle,$$
we easily conclude that
$$a\sim a^n\sim ab \sim b^{o(a)} \sim b,$$
which shows that the vertices $a$ and $b$ belong to the same
component of $\mathcal{P}^\ast(G)$, as claimed.
\end{proof}

\begin{lm}\label{cyclic} Let $G=\langle x\rangle$ be a cyclic
group of order $n$, $D=\vec{\mathcal{P}}(G)$ and
$\Gamma=\mathcal{P}(G)$. Then the out-degree and in-degree of
$x^m\in G$ in the direct graph $D$ is given by
$$d^+_D(x^m)=o(x^m)-1=\frac{n}{(m, n)}-1 \ \ \ \mbox{and} \ \ \
d^-_D(x^m)=\sum_{d|(m,n)}\phi(\frac{n}{d}),$$ respectively.
Furthermore, the degree of $x^m\in G$ in the power graph $\Gamma$
is given by
$$d_\Gamma(x^m)=\frac{n}{(m,
n)}-1+\sum_{d|(m,n) \atop d\neq (m,n)}\phi(\frac{n}{d}).$$
\end{lm}
\begin{proof} Suppose $x^k$ is an arbitrary element in $G$. Then, by the definition we have
$$\begin{array}{lll} x^m\sim x^k & \Longleftrightarrow & x^m\in
\langle x^k\rangle \ \ \ {\rm or} \ \ \ x^k\in \langle x^m\rangle
\\[0.2cm] & \Longleftrightarrow &  x^m=(x^k)^r  \ {\rm or} \ \ x^k=(x^m)^s  \ \ {\rm for \ some} \ r, s
\\[0.2cm] & \Longleftrightarrow &  kr\equiv m\pmod{n}  \ \ {\rm or} \ \ ms\equiv k\pmod{n}
\\[0.2cm] & \Longleftrightarrow &  (k, n)|m  \ \ {\rm or} \ \ (m, n)|k.\\
\end{array}$$

First, we determine the number of elements $x^k$ such that $1\leq
k\leq n$ and $(k, n)|m$, which is equal to $d^-_D(x^m)$.  Suppose
$d_1=1, d_2, \ldots, d_t$ are common divisors of $m$ and $n$.
Now, using the fact that $(k, n)=d_i$ if and only if
$(\frac{k}{d_i}, \frac{n}{d_i})=1$, we conclude that
\begin{equation}\label{rah1}
d^-_D(x^m)=\sum_{d|(m, n)}\phi(\frac{n}{d}).
\end{equation}
Next, we determine the number of elements $x^k$ such that $1\leq
k\leq n$ and $(m, n)|k$, which is equal to $d^-_D(x^m)$. In fact,
from $(m, n)|k$ and $ms\equiv k\pmod{n}$ (for some $s$), it
follows that $s=1, 2, 3, \ldots, o(x^m)=\frac{n}{(m, n)}$.
Therefore, since $D$ is loopless, we obtain
\begin{equation}\label{rah2}
d^+_D(x^m)=o(x^m)-1=\frac{n}{(m, n)}-1.
\end{equation}
Evidently $d_\Gamma(G)\leq d^-_D(x^m)+d^+_D(x^m)$. In the sequel
to calculate the degree of $x^m$ in $\Gamma$, we have to count
the bidirected arcs in $\vec{\mathcal{P}}(\langle x^m \rangle)$
(see Figure 1). In fact, the number of such arcs is precisely
$$\phi(o(x^m))=\phi(\frac{n}{(m, n)}),$$ and so from Eqs.
(\ref{rah1}) and (\ref{rah2}) we obtain
$$\begin{array}{lll} d_{\Gamma}(x^m)&=&\Big(d^-_D(x^m)-\phi(\frac{n}{(m, n)})\Big)+d^+_D(x^m)\\[0.3cm]
&=& o(x^m)-1-\phi(\frac{n}{(m, n)})+\sum\limits_{d|(m,
n)}\phi(\frac{n}{d})\\[0.3cm]
&=& \frac{n}{(m, n)}-1+\sum\limits_{d|(m,n) \atop d\neq
(m,n)}\phi(\frac{n}{d}).\\
\end{array}$$ This completes the proof.
\end{proof}

\vspace{1.5cm} \setlength{\unitlength}{4mm}
\begin{picture}(10,8)(-3,1)
\linethickness{0.7pt} %
\put(12,5){\circle*{0.3}}%
\put(10.5,8){\circle*{0.3}}%
\put(12,11){\circle*{0.3}}%
\put(15,8){\circle*{0.3}}%
\put(18,5){\circle*{0.3}}%
\put(19.5,8){\circle*{0.3}}%
\put(18,11){\circle*{0.3}}%
\put(14.7,8.6){\small$x^m$}%
\put(15,8){\vector(1,1){2.9}}%
\put(15,8){\vector(1,0){4.4}}%
\put(15,8){\vector(1,-1){2.9}}%
\put(12,5){\vector(1,1){2.9}}%
\put(10.5,8){\vector(1,0){4.4}}%
\put(12,11){\vector(1,-1){2.9}}%
\put(15,8){\vector(2,1){3.9}}%
\put(15,8){\vector(2,-1){3.9}}%
\put(19,10){\circle*{0.3}}%
\put(19,6){\circle*{0.3}}%
\put(19.5,8){\vector(-1,0){4.4}}%
\put(18,5){\vector(-1,1){2.9}}%
\put(18.2,11.1){\small $1$}%
\put(19.4,10){\small $x^{2m}$}%
\put(19.9,7.9){\small $x^{sm}$ \ $(s, o(x^m))=1$}%
\put(19.4,5.6){\small $x^{(o(x^m)-2)m}$}%
\put(18,4){\small $x^{(o(x^m)-1)m}$}%
\put(11,11.1){\small $x^{k_1}$}%
\put(9.3,8){\small $x^{k_2}$}%
\put(10.7,4.3){\small$x^{k_d}$}%
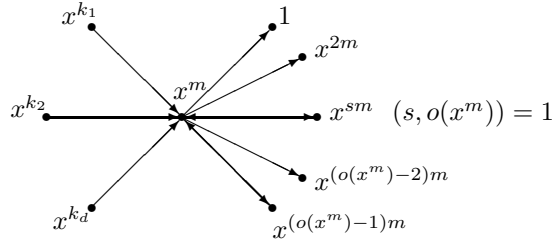
\put(-2,1){{\small {\bf Figure 1.} \ The out-degree and in-degree
of vertex $x^m$ in $D=\vec{\mathcal{P}}(\langle x\rangle)$. Here
$d=d^-_D(x)-\phi(o(x^m))$.}}
\end{picture}
\vspace{0.5cm}

The following corollary is an immediate consequence of Lemma
\ref{cyclic}.
\begin{corollary}\label{corollary-cyclic}
All non-trivial elements of a cyclic group with the same orders
have the same degrees in its power graph.
\end{corollary}

\begin{lm}\label{disconnected-srg}
Suppose $\Gamma$ is a strongly regular graph with parameters $(n,
k, \lambda, \mu)$. Then, the following are equivalent:

$(a)$ $\Gamma$ is disconnected.

$(b)$ $\Gamma$ is isomorphic with $tK_{k+1}$, for some $t>1$.
\end{lm}
\begin{proof} See Lemma 10.1.1 in \cite{Godsil-Royle}. \end{proof}
\subsection{Power Graphs and Connectivity} For a finite group $G$,
the power graph $\mathcal{P}(G)$ is always connected, because the
identity element of $G$ is adjacent to all other vertices of
$\mathcal{P}(G)$. In this section, we will focus our attention on
the connectivity of $\mathcal{P}^\ast(G)$, the subgraph of
$\mathcal{P}(G)$ obtained by deleting the vertex 1 (the identity
element of $G$). In the next result, we provide necessary and
sufficient conditions for the subgraph $\mathcal{P}^\ast(G)$ to
be connected, when $G$ is a $p$-group for some prime $p$.
\begin{lm}\label{lemma3} Let $G$ be a finite $p$-group. Then
$\mathcal{P}^\ast(G)$ is connected if and only if $G$ has unique
minimal subgroup.
\end{lm}
\begin{proof} Let $G$ be a finite $p$-group with a unique minimal
subgroup of order $p$ for a prime $p$, say $\langle x\rangle$.
Now let $g\in G\setminus \{1\}$ be an arbitrary $p$-element.
Then, it is clear that $\langle x\rangle\subseteq \langle
g\rangle$, and hence $x\sim g$. This shows that
$\mathcal{P}^\ast(G)$ is connected, as required.

Conversely, assume that $G$ is a finite $p$-group such that
$\mathcal{P}^\ast(G)$ is connected. We claim that $G$ possesses a
unique minimal subgroup of order $p$. Suppose on the contrary that
there are two distinct minimal subgroups of $G$ of order $p$, say
$\langle x \rangle\neq \langle y\rangle$. Since
$\mathcal{P}^\ast(G)$ is connected, there is a path between the
vertices $x$ and $y$. Assume that $$x=x_0\sim x_1 \sim x_2 \sim
\ldots \sim x_n=y,$$ is a path with the least length from $x$ to
$y$. Certainly $n\geq 2$ and $x\nsim x_i$, for each $i=2, 3,
\ldots, n$. Since $x\sim x_1$, by the definition we have $x\in
\langle x_1\rangle$ or $x_1\in \langle x\rangle$. We only
consider the first case, and the second one goes similarly. Since
$x_1\sim x_2$, it follows that $x_1\in \langle x_2\rangle$ or
$x_2\in \langle x_1\rangle$. If $x_1\in \langle x_2\rangle$, then
$x\in \langle x_2\rangle$, which is a contradiction. Therefore
$x_2\in \langle x_1\rangle$. But then $x, x_2\in \langle
x_1\rangle$ and since $\langle x_1\rangle$ is a $p$-group, Lemma
\ref{elementary-results} $(i)$ shows that $\langle
x_1\rangle\setminus\{1\}$ is a clique in $\mathcal{P}^\ast(G)$.
Hence $x\sim x_2$, an impossible. The proof of this lemma is
complete.
\end{proof}

{\bf Remark 1.} The finite $p$-groups with a unique minimal
subgroup of order $p$ are well characterized: {\em If $G$ is a
$p$-group which has a unique minimal subgroup of order $p$, then
$G$ is either a cyclic group or a generalized quaternion group}
\cite[Theorem 5.4.10. $(ii)$]{Gor}. Notice that a generalized
quaternion group is not necessarily a 2-group.

An immediate consequence of Lemma \ref{lemma3} is the following.

\begin{corollary}\label{corollary1} Let $G$ be a finite $p$-group. Then
$\mathcal{P}^\ast(G)$ is connected if and only if $G$ is either
cyclic or generalized quaternion.
\end{corollary}
{\bf Some Examples.} The generalized quaternion group $Q_{4n}$, is
defined by
$$Q_{4n}=\langle x, y \ | \ x^{2n}=1, y^2=x^n, x^y=x^{-1}\rangle.$$ For $n=2$,
there is another well-known representation: $$Q_8=\{\pm 1, \ \pm
i, \  \pm j, \  \pm k \ | \ i^2=j^2=k^2=-1, ij=k, jk=i, ki=j\}.$$
We recall some elementary facts about generalized quaternion
groups below without proof, because they can be calculated in
straightforward ways.

\begin{itemize}
\item[$\bullet$] Every element $g\in Q_{4n}$ can be written uniquely as
$g=x^iy^j$ where $0\leq i\leq 2n-1$ and $0\leq j\leq 1$. In
particular, the order of $Q_{4n}$ is exactly $4n$.
\item[$\bullet$]  We have $$\mu(Q_{4n})=\left\{\begin{array}{ll} \{4, 2n\} & n \ \mbox{odd},\\[0.2cm]
\{2n\} & n \ \mbox{even}. \end{array} \right.$$ In fact, it is
easy to see that  $$o(x^iy^j)=\left\{\begin{array}{ll}
\frac{2n}{(i, 2n)} & j=0, \ 1\leq i\leq
2n,\\[0.2cm]
4 & j=1, \ 1\leq i\leq 2n. \end{array} \right.$$
\item[$\bullet$] The center of $Q_{4n}$ is $\langle x^n\rangle=\langle y^2\rangle\cong \mathbb{Z}_2.$
\item[$\bullet$] The generalized quaternion group $Q_{4n}$ has a unique
minimal subgroup if and only if $n$ is a power of $2$.
\item[$\bullet$] If $n$ is a power of $2$, then the
power graph $\mathcal{P}^\ast(Q_{4n})$ has the following form:
$$\mathcal{P}^\ast(Q_{4n})=K_1\vee (K_{2n-2}\oplus
\underbrace{K_2\oplus K_2\oplus \cdots \oplus K_2}_{n {\rm
-times}}).$$ Actually, $\mathcal{P}^\ast(Q_{4n})$ consists of a
complete graph on $2n-1$ vertices and $n$ triangles sharing a
common vertex (the unique involution). Note that, for every
element $g\in Q_{4n}\setminus \{1\}$, the subgroup $\langle g
\rangle$ contains the unique involution $x^n$, and so $x^n\sim g$
in $\mathcal{P}^\ast(Q_{4n})$. On the other hand, all vertices of
$\mathcal{P}^\ast(G)$ with degree $|G|-2$ lie in the center of
$G$, hence in $\mathcal{P}^\ast(Q_{4n})$ the unique involution
$x^n$ is the {\em only} vertex of degree $4n-2$.  For instance,
the power graph $\mathcal{P}^\ast(Q_{8})$ is depicted in Figure 2.

\vspace{1.5cm} \setlength{\unitlength}{4mm}
\begin{picture}(1,9)(5,1)
\linethickness{0.3pt} %
\put(12,8){\circle*{0.5}}%
\put(15,12){\circle*{0.5}}%
\put(19,12){\circle*{0.5}}%
\put(17,8){\circle*{0.5}}%
\put(15,12){\line(1,0){4}}%
\put(15,12){\line(1,-2){2}}%
\put(17,8){\line(1,2){2}}%
\put(17,8){\circle*{0.5}}%
\put(22,8){\circle*{0.5}}%
\put(13.7,4.8){\circle*{0.5}}%
\put(20.3,4.8){\circle*{0.5}}%
\put(12,8){\line(1,0){10}}%
\put(17,8){\line(1,-1){3.3}}%
\put(17,8){\line(-1,-1){3.3}}%
\put(12,8){\line(1,-2){1.65}}%
\put(22,8){\line(-1,-2){1.65}}%
\put(14.8,12.6){\small$i$}%
\put(18.6,12.6){\small$-i$}%
\put(17.5,8.2){\small$-1$}%
\put(10.3,7.8){\small$-k$}%
\put(13.3,3.7){\small$k$}%
\put(19.7,3.7){\small$-j$}%
\put(22.5,7.8){\small$j$}%
\put(28,4){\circle*{0.5}}%
\put(36,4){\circle*{0.5}}%
\put(32,12){\circle*{0.5}}%
\put(28,4){\line(1,2){4}}%
\put(32,12){\line(1,-2){4}}%
\put(28,4){\line(1,0){8}}%
\put(28,4){\line(2,1){6.4}}%
\put(36,4){\line(-2,1){6.4}}
\put(29.6,7.2){\circle*{0.5}}%
\put(34.5,7.2){\circle*{0.5}}%
\put(29.6,7.2){\line(1,0){4.8}}%
\put(31.8,12.5){\small$x^3$}%
\put(28.7,7.2){\small$x$}%
\put(35,7.2){\small$x^5$}%
\put(26.8,4){\small$x^2$}%
\put(36.5,4){\small$x^4$}%
\put(8,1.5){\small {\small \bf Figure 2.} \ The power graph $\mathcal{P}^\ast(Q_8)$}%
\put(25,1.5){{\small \bf Figure 3.} \ The power graph \small$\mathcal{P}^\ast(\mathbb{Z}_6)$}%
\end{picture}
\vspace{0.5cm}
\end{itemize}

We next note that, if there are some vertices in the graph
$\mathcal{P}^\ast(G)$ which are adjacent to all other vertices,
then $\mathcal{P}^\ast(G)$ is connected. In particular, if $G$ is
a finite cyclic group with $|G|>1$, then $\mathcal{P}^\ast(G)$ is
connected, because each generator of $G$ is adjacent to all other
nontrivial elements of $G$ in $\mathcal{P}^\ast(G)$. In what
follows, a necessary and sufficient condition is established for
the existence of such vertices in $\mathcal{P}^\ast(G)$. The
proof of the next lemma is similar to the proof of Proposition 4
in \cite{C2} and is included for the sake of completeness.
\begin{lm} If $G$ is a finite group, then $\mathcal{P}^\ast(G)$ contains a vertex which is joined
to all other vertices if and only if $G$ is either cyclic or
generalized quaternion.
\end{lm}
\begin{proof} We need only prove the necessity part.
(Note that, the unique involution in a generalized quaternion
$2$-group is joined to all other vertices in its power graph, and
in the case when $G$ is a cyclic group every generator of $G$ has
the same property.) In the sequel, we will consider separately
two cases: $G$ is a $p$-group or not.

{\em Case $1$.} {\em $G$ is a $p$-group.}

In this case, since $\mathcal{P}^\ast(G)$ is connected, the
conclusion follows immediately from Corollary \ref{corollary1}.

{\em Case $2$.} {\em $G$ is not a $p$-group.}

In this case, we claim that $G$ is always cyclic. Assume to the
contrary that $G$ is not a cyclic group. Let $x\in
G\setminus\{1\}$ be a vertex in $\mathcal{P}^\ast(G)$ which is
joined to all others. Clearly $\langle x\rangle$ is a proper
subgroup of $G$. Suppose $y\in G\setminus \langle x\rangle$.
Since $x\sim y$, it follows by the definition that $\langle
x\rangle\subsetneqq \langle y \rangle$, and so
\begin{equation}\label{ee1}
\ o(x)|o(y) \ \mbox{while} \ o(x)<o(y).
\end{equation}
Moreover, from part $(c)$ of Lemma \ref{elementary-results}, we
deduce that $\pi(o(x))=\pi(G)$ and so $x$ is not a $p$-element
for some prime $p\in \pi(G)$.

Suppose now that $\pi(G)=\{p_1, p_2, \ldots, p_k\}$. Let $\langle
g_i \rangle$ be a {\em maximal} cyclic $p_i$-subgroup of $G$ for
each $1\leq i\leq k$. Since $x\sim g_i$, $g_i\in \langle x\rangle$
otherwise $x$ would be a $p_i$-element of $G$, which is a
contradiction. Therefore, $o(x)$ is divisible by $o(g_i)$, for
each $1\leq i\leq k$, which shows by the maximality property of
$o(g_i)$ that $\mu(G)=\{o(x)\}$. But then $o(y)|o(x)$, which is a
contradiction from Eq. (\ref{ee1}).
\end{proof}

\begin{lm}\label{lemma4} If $G$ is a finite group such that $|\pi(Z(G))|\geq 2$, then
$\mathcal{P}^\ast(G)$ is connected.
\end{lm}
\begin{proof} Since $Z(G)$ is abelian, it follows that
$|\mu(Z(G))|=1$. Let $\mu(Z(G))=\{m\}$ and $o(g)=m$ for some $g\in
Z(G)$. In the sequel, we will show that for each element $x\in
G\setminus \{1\}$ there is a path between $x$ and $g$, which means
that $\mathcal{P}^\ast(G)$ is connected. we distinguish two cases:

{\em Case $1$.} $\pi(o(x))=\pi(m)$.

In this case, since $|\pi(m)|\geq 2$, there are two distinct
primes $p$ and $q$ in $\pi(m)$ such that $o(x^r)=p$ and
$o(g^s)=q$ for some natural numbers $r$ and $s$. Evidently $x^r$
and $g^s$ are elements of coprime order and they commute, hence
from Lemma \ref{FACT} we deduce that $x^r$ and $g^s$ lie in the
same component of $\mathcal{P}^\ast(G)$. On the other hand $x\sim
x^r$ and $g\sim g^s$, hence $x$ and $g$ lie in the same component
of $\mathcal{P}^\ast(G)$ too.

{\em Case $2$.} $\pi(o(x))\neq \pi(m)$.

In this case, there exists a prime $q\in \pi(o(x))\setminus
\pi(m)$ (or $q\in \pi(m)\setminus \pi(o(x))$). Assume that $q\neq
p\in \pi(m)$. Again there are two elements $x^r$ and $g^s$ for
some natural numbers $r$ and $s$ such that $o(x^r)=q$ and
$o(g^s)=p$. Using a similar argument as previous paragraph we see
that these elements belong to the same component of
$\mathcal{P}^\ast(G)$ and since $x\sim x^r$ and $g\sim g^s$,
hence $x$ and $g$ lie in the same connected component of
$\mathcal{P}^\ast(G)$ too.
\end{proof}

\begin{lm} Let $G$ be a finite group with $|\pi(G)|\geq 2$ and let its center is a $p$-group
for some prime $p\in \pi(G)$. Then, the power graph
$\mathcal{P}^\ast(G)$ is connected if and only if for every
non-central element $x$ of order $p$ there exists a non
$p$-element $g$ such that $x\sim g$ in $\mathcal{P}^\ast(G)$.
\end{lm}
\begin{proof} ($\Longrightarrow$) Suppose to the contrary that there
exists a non-central $p$-element $x$ of order $p$ such that for
every non $p$-element $g$, $x\nsim g$. Let $z\in Z(G)$ be an
element of order $p$. In the sequel we will show that there is no
path between $x$ and $z$. Suppose the contrary that there is a
path between $x$ and $z$, say $P$, with minimum possible length:
$$P: \ \ \ x\sim c_1\sim c_2\sim c_3\sim \cdots \sim c_k\sim z.$$
We claim that $x\sim z$. If so then we obtain $\langle x
\rangle=\langle z\rangle$, which is a contradiction because
$x\not \in Z(G)$.

Therefore, we assume that $k\geq 1$. Note that by our assumption
$c_1$ is a $p$-element of $G$. Since $x\sim c_1$, by the
definition $x\in \langle c_1\rangle$ or $c_1\in \langle
x\rangle$. We distinguish two cases separately.

{\em Case $1$.} $x\in \langle c_1\rangle$.

In this case, since $c_1\sim c_2$ and $P$ has minimum length, it
follows that $c_2\in \langle c_1\rangle$. Since $\langle
c_1\rangle$ is a $p$-group, $\mathcal{P}^\ast(\langle c_1\rangle)$
is a complete graph by Lemma \ref{complete-CSS}. Hence $x\sim
c_2$, which is a contradiction by minimality of length of $P$.

{\em Case $2$.} $c_1\in \langle x\rangle$.

This case can be proven in a similar way as previous case.

($\Longleftarrow$) Since $Z(G)$ is a proper sugroup of $G$, it is
enough to show that for each $g\in G\setminus Z(G)$, there is a
path from $g$ to each non-trivial element of $Z(G)$. Therefore,
we assume that $g\in G\setminus Z(G)$ is an arbitrary element.
Again, we will consider two cases separately.

{\em Case $1$.} $g$ is not a $p$-element.

Suppose $q\in \pi(G)\setminus \{p\}$ divides $o(g)$ and
$o(g^l)=q$ for some positive integer $l$. Then by Lemma
\ref{FACT},  there is a path between $g^l$ and every non-trivial
element in $Z(G)$, and so there is a path between $g$ and all
non-trivial elements in $Z(G)$.

{\em Case $2$.} $g$ is a $p$-element.

Suppose $o(g^l)=p$ for some positive integer $l$. Assume first
that $g^l\in Z(G)$. In this case, if $x$ is a $p'$-element of
$G$, then there is a path between $x$ and $g^l$ by Lemma
\ref{FACT}. On the other hand, $g\sim g^l$ and from Case 1 there
is a path between $x$ and every non-trivial element in $Z(G)$, as
required. Assume next that $g^l\not \in Z(G)$. Now, from our
assumption there is a non $p$-element in $G$, say $h$, such that
$g^l\sim h$. Again $g\sim g^l \sim h$ and by Case 1 there is a
path between $h$ and so $g$ and all non-trivial elements in
$Z(G)$. The proof of lemma is now complete.
\end{proof}
\subsection{The Relation Between Power Graph and the Other Graphs}
\begin{lm} Let $G$ be a finite group. If the power graph $\mathcal{P}^\ast(G)$ is
connected, then the prime graph $\Gamma(G)$ is connected too.
\end{lm}
\begin{proof}
Let $x$ and $y$ be two arbitrary adjacent vertices in the power
graph $\mathcal{P}^\ast(G)$. Then by the definition we have $x\in
\langle y\rangle$ or $y\in \langle x\rangle$, and so $o(x)|o(y)$
or $o(y)|o(x)$. However, this means that all primes in
$\pi(o(x))\cup \pi(o(y))$ lie in the same connected component of
the prime graph $\Gamma(G)$. Since $\mathcal{P}^\ast(G)$ is
connected we may proceed successively in this manner, eventually
showing that all primes in $$\pi(G)=\bigcup_{x\in G} \pi(o(x)),$$
lie in the same connected component of $\Gamma(G)$, and hence
$\Gamma(G)$ is connected, as required. \end{proof}
\begin{lm}\label{disconnected-srg}
Let $G$ be a group. Then $\mathcal{P}^\ast(G)$ is a spanning
subgraph of $\Delta(G)$.
\end{lm}
\begin{proof}
First of all, one can easily see that both graphs
$\mathcal{P}^\ast(G)$ and $\Delta(G)$ have the same vertex set,
i.e., $G\setminus \{1\}$. Also from the definition of the power
graph, it follows that for any $x, y\in G\setminus \{1\}$, they
are adjacent in $\mathcal{P}^\ast(G)$ if $x\in \langle y\rangle$
or $y\in \langle x\rangle$, and in both cases they commute which
shows that they are adjacent in $\Delta(G)$. Thus
$\mathcal{P}^\ast(G)$ is an induced subgraph of $\Delta(G)$.
\end{proof}
\section{Recognizing Some Almost Simple Groups by Power Graph}
In \cite{C2}, Cameron showed that the power graph of a finite
group determines its spectrum.
\begin{lm}[Corollary 3, \cite{C2}]\label{l1}
Two finite groups whose power graphs are isomorphic have the same
numbers of elements of each order. In particular, they have the
same spectra.
\end{lm}

In 1987, the third author of the paper, Shi posed that every
finite simple group is uniquely determined by its spectrum and
order in the class of all groups (see \cite{S1}). In fact, this
conjecture is Question 12.39 in the Kourovka Notebook \cite{Un},
and is stated as
follows:\\[0.3cm]
{\bf  Shi's Conjecture.} {\em A finite group and a finite simple
group are isomorphic if they have the same orders
and spectra.}\\[0.3cm]
\indent The answer to Shi's Conjecture is obviously `yes' for
abelian simple groups. In a series of papers \cite{Cao, S1, S4,
S5, S3, S2, Vgm, Xu}, an affirmative answer was given for all
nonabelian simple groups. Thus, Shi's Conjecture is confirmed and
the following proposition holds.
\begin{proposition}\label{p1}
If $S$ is a finite simple group, and $G$ is a finite group with
$\pi_e(G)=\pi_e(S)$ and $|G|=|S|$, then $G \cong S$.
\end{proposition}

Besides the finite simple groups one can characterize some
non-solvable groups using their spectra and orders. For instance,
Bi proved the following result concerning symmetric groups (see
\cite{Bi}).
\begin{proposition}\label{p2}
Let $G$ be a finite group and $n\geq 3$ an integer. Then $G\cong
S_n$ if and only if $|G|=|S_n|$ and $\pi_e(G)=\pi_e(S_n)$.
\end{proposition}

In \cite{Kh, Shen}, it is proved that Shi's Conjecture is valid
for the automorphism groups of sporadic simple groups. In fact,
we have the following.
\begin{proposition}\label{p3}
Let $G$ be a finite group and $A$ the automorphism group of a
sporadic simple group. Then $G\cong A$ if and only if $|G|=|A|$
and $\pi_e(G)=\pi_e(A)$.
\end{proposition}

Given a group $M$, denote by $h_{\mathcal{P}}(M)$ the number of
isomorphism classes of groups $G$ such that
$\mathcal{P}(G)\cong\mathcal{P}(M)$. A group $M$ is called
$k$-fold $\mathcal{P}$-characterizable if $h_{\mathcal{P}}(M)=k$.
Usually, a $1$-fold $\mathcal{P}$-characterizable group is simply
called $\mathcal{P}$-characterizable. We are now ready to return
to Question \ref{q1}. By combining Lemma \ref{l1} and
Propositions \ref{p1}-\ref{p3}, we obtain the following theorem.

\begin{theorem}\label{th2} Let $M$ be one of the following groups:

$(a)$ A finite simple group,

$(b)$ A symmetric group,

$(c)$ The automorphism group of a sporadic simple group.

Then, $h_{\mathcal{P}}(M)=1$, in other words, the group $M$ is
$\mathcal{P}$-characterizable.
\end{theorem}
\begin{proof} The proof follows in a straightforward way from
Lemma \ref{l1} and Propositions \ref{p1}, \ref{p2} and \ref{p3}.
Note that, by Lemma \ref{l1}, two groups whose power graphs are
isomorphic have, in fact, the same order and spectra. \end{proof}

It is well-known that for each positive integer $n$ there are
only {\em finitely} many non-isomorphic groups of order $n$
normally denoted by $\nu(n)$. In fact, the number of $n\times n$
arrays with entries chosen from a set of size $n$ is $n^{n^2}$.
So certainly this is an upper bound for $\nu(n)$. For another
upper bound we have $\nu(n)\leq (n!)^{\log_{2}n}$(see \cite{r},
page 109). Therefore the following result follows immediately.
\begin{theorem}\label{th1}
All finite groups are $k$-fold $\mathcal{P}$-characterizable for
some natural number $k$.
\end{theorem}

\section{The Power Graph of Finite Group as an Strongly Regular
Graph}

A {\em strongly regular graph} (henceforth SRG) with parameters
$(n, k, \lambda, \mu)$, which will be denoted by ${\rm srg} (n,
k, \lambda, \mu)$, is a regular graph on $n$ vertices of valency
$k$ such that each pair of adjacent vertices has exactly
$\lambda$ common neighbours, and each pair of non-adjacent
vertices has exactly $\mu$ common neighbours.

\begin{theorem}\label{th3}
Let $G$ be a nontrivial finite group. Then the following
conditions are equivalent.

$(a)$ $\mathcal{P}^\ast (G)$ is a strongly regular graph.

$(b)$ $G$ is a $p$-group of order $p^m$ for which $\exp(G)=p$ or
$p^m$, for some prime $p$.
\end{theorem}
\begin{proof} $(a)\Longrightarrow
(b).$ Let $G$ be a group such that $\mathcal{P}^\ast (G)$ is a
strongly regular graph with parameters $(n, k, \lambda, \mu)$. We
distinguish two cases separately according to it is connected or
disconnected.

{\bf Case 1.} {\em  $\mathcal{P}^\ast (G)$ is disconnected.}

In this case, by Lemma \ref{disconnected-srg}, we have
$$\mathcal{P}^\ast(G)\cong K_{k+1}\oplus K_{k+1}\oplus\cdots\oplus K_{k+1}=tK_{k+1},$$
for some $t>1$.

Suppose first that $x_1\in G\setminus \{1\}$ such that
$o(x_1)=\max \pi_e(G)$. Then by Lemma \ref{elementary-results}
$(b)$, $d(x_1)=o(x_1)-2$. Therefore, $o(x_1)=k+2$ and the
connected component containing $x_1$ is $\mathcal{P}^\ast(\langle
x_1\rangle)$.

Suppose next that $x_2\in G\setminus \langle x_1\rangle$ with
$o(x_2)=\max \{o(g) \ | \ g\in G\setminus \langle x_1\rangle\}$.
Again, by Lemma \ref{elementary-results} $(b)$, we conclude that
$d(x_2)=o(x_2)-2$. Hence $o(x_2)=k+2$, and the connected
component containing $x_2$ is $\mathcal{P}^\ast(\langle
x_2\rangle)$.

We continue this process until we obtain the disjoint cyclic
subgroups $\langle x_1\rangle$, $\langle x_2\rangle$, $\ldots$,
$\langle x_t\rangle$, such that
\begin{equation}\label{e1}
G=\langle x_1\rangle \cup \langle x_2\rangle \cup \dots \cup
\langle x_t\rangle.\end{equation} Evidently
$o(x_1)=o(x_2)=\ldots=o(x_t)=k+2$.

In the sequel, we show that $G$ is a $p$-group for some prime $p$.
For this purpose, it is enough to show that $|\pi(k+2)|=1$.
Assume to the contrary that $|\pi(k+2)|>1$ and let $p, q\in
\pi(k+2)$ be two distinct primes. Further we assume that
$x_i^{k_1}$ and $x_i^{k_2}$ are two vertices in the $i$th
connected component of order $p$ and $q$, respectively. Since
each connected component is clique, $x_i^{k_1}\sim x_i^{k_2}$,
which is a contradiction by Lemma \ref{elementary-results} $(c)$.

Now we show that $\exp(G)=p$. To do this, it is enough to show
that $o(x_1)=p$. Assume to the contrary that $o(x_1)=p^l$, where
$l>1$. Since $G$ is a nontrivial $p$-group, $Z(G)>1$ and so we
can choose an element of order $p$, say $z$, in $Z(G)\setminus
\{1\}$. From (\ref{e1}), it follows that $z\in \langle
x_i\rangle$ for some $i$. Since $t>1$, we can consider the
element $zx_j$, where $j\neq i$. Evidently, $zx_j\notin \langle
x_i\rangle\cup \langle x_j \rangle$. Therefore $t\geq 3$, and so
$zx_j\in \langle x_k\rangle$, where $k\neq i, j$. But then, we
obtain
$$(zx_j)^p=x_j^p\in\langle x_k\rangle\setminus \{1\},$$
since $o(x_j)=o(x_1)=p^l>p$. This means that there exists a path
between $x_j$ and $x_k$, which is a contradiction.

{\bf Case 2.} {\em  $\mathcal{P}^\ast (G)$ is connected.}

In this case, we first claim that:

\centerline {\em  \ \ \ \ \ \ \ \ \ \ \ \ ``$G$ is a cyclic group
if and only if $G$ is a $p$-group for some prime $p$." \ \ \ \ \
\ \ \ \ \ \ \ $(\dagger)$ }

First, if $G$ is a cyclic group, say $G=\langle x\rangle$, then
$o(x)=\max \pi_e(G)$ and from Lemma \ref{elementary-results}
$(b)$, we have
$$d(x)=o(x)-2=|G|-2,$$ hence $\mathcal{P}^\ast (G)$ is a
$(|G|-2)$-regular graph. Thus $\mathcal{P}^\ast (G)$ and so
$\mathcal{P}(G)$ is a complete graph and by Lemma
\ref{complete-CSS} it follows that $G$ is a $p$-group for some
prime $p$.

Conversely, we assume that $G$ is a $p$-group of order $p^m$ and
$x\in G\setminus \{1\}$ such that $o(x)=p^l\in \mu(G)$. It is
enough to show that $l=m$. Again by Lemma \ref{complete-CSS}, we
see that $\mathcal{P}(\langle x\rangle)$ and so $\mathcal{P}^\ast
(\langle x\rangle)$ is a complete graph. Hence $\mathcal{P}^\ast
(G)$  is a connected $(p^l-2)$-regular graph which includes a
clique of size $p^l-1$, this forces $G=\langle x\rangle$.

In what follows, we will show that $G$ is a $p$-group for some
prime $p$, which implies by $(\dagger)$ that $G$ is cyclic, as
required. Assume that $x\in G\setminus \{1\}$ such that
$o(x)=\max \pi_e(G)$. Then $d(x)=o(x)-2$ and so $\mathcal{P}^\ast
(G)$ is a $(o(x)-2)$-regular graph. Assume first that
$|\pi(o(x))|\geq 2$ and let $p, q\in \pi(o(x))$ be two distinct
primes. Since $o(x)>2$, the cyclic group $\langle x\rangle$ has
at least two generators, say $x$ and $x^t$ for some $t$. Actually
$x$ and $x^t$ are two adjacent vertices with $d(x)=d(x^t)$ and
each of them joint to all vertices in $\langle x\rangle\setminus
\{1, x, x^t\}$, which forces $\lambda=o(x)-3$. Suppose now that
$x^r$ and $x^s$ are two elements in $\langle x\rangle$ of order
$p$ and $q$, respectively. Evidently $x^r\nsim x^s$. On the other
hand, the vertices $x$ and $x^r$ are adjacent and hence they have
$o(x)-3$ common neighbours. However, since $N(x)=\langle
x\rangle\setminus \{1, x\}$, we conclude that $$ N(x)\cap
N(x^r)\subseteq N(x)\setminus \{x^s\}=\langle x\rangle\setminus
\{1, x, x^s\},$$ and since $|N(x)\cap N(x^r)|=o(x)-3$, we get
$N(x)\cap N(x^r)=\langle x\rangle\setminus \{1, x, x^s\}$, which
is a contradiction because $x^r\notin N(x^r)$.

Next, suppose that $|\pi(o(x))|=1$. In this case, $\langle x
\rangle$ is a $p$-group for some prime $p$, and $\mathcal{P}^\ast
(G)$ is $(o(x)-2)$-regular. Moreover, by Lemma \ref{complete-CSS},
$\langle x\rangle$  is a clique in $\mathcal{P}^\ast(G)$, which
forces $G=\langle x\rangle$. Therefore $G$ is a $p$-group, and
the proof is complete.

$(b)\Longrightarrow (a)$. The proof is straightforward.
\end{proof}

\section{The Power Graph Which is Bipartite or Planar}
An {\em independent set} in a graph is a set of pairwise
nonadjacent vertices. A graph $\Gamma$ is called {\em bipartite}
if whose vertex set can be partitioned into two independent sets
called partite sets. In the following result we recognize the
groups $G$ for which the graph $\mathcal{P}^\ast (G)$ is
bipartite.
\begin{theorem}\label{2partite}
Let $G$ be a nontrivial finite group. Then the following
conditions are equivalent.

$(a)$ $\mathcal{P}^\ast (G)$ is a bipartite graph.

$(b)$ $\pi_e(G)\subseteq \{1, 2, 3\}$.
\end{theorem}
\begin{proof} $(a) \Longrightarrow (b)$.
Let $\mathcal{P}^\ast (G)$ be a bipartite graph. Clearly, if $n\in
\pi_e(G)$, then $G$ always has at least $\phi(n)$ elements of
order $n$, where $\phi(n)$  signifies the Euler's totient
function. Suppose now that $G$ contains an element of order $\geq
4$, say $x$. Certainly $\phi(o(x))>1$, so the cyclic group
$\langle x\rangle$ has at least two generators, say $x$ and
$x^t$. Now, we have the $3$-cycle $$x\sim x^j\sim x^t\sim x,$$
where $j\neq 1, t$. This shows that $\mathcal{P}^\ast (G)$ cannot
be a bipartite graph, which is a contradiction.

$(b) \Longrightarrow (a)$. Let $G$ be a nontrivial finite group
such that $\pi_e(G)\subseteq \{1, 2, 3\}$. In the case when
$\pi_e(G)= \{1, 2\}$, $G$ is an elementary abelian $2$-group and
by the definition of power graph, it is easy to see that the
graph $\mathcal{P}^\ast (G)$ consists of only isolated vertices
and hence it can be considered as a bipartite graph. On the other
hand, if $3\in \pi_e(G)\subseteq \{1, 2, 3\}$, then $3=\max
\pi_e(G)$ and for all elements of order $3$, say $x$, we have the
singleton edge $x\sim x^2$ and so $d(x)=1$. In fact, the power
graph $\mathcal{P}^\ast (G)$ is isomorphic to a graph as in the
following:
$$\mathcal{P}^\ast (G)\cong \underbrace{K_1\oplus K_1\oplus
\cdots \oplus K_1}_{\rm elements \ of \ order \ 2}\oplus
\underbrace{K_2\oplus K_2\oplus \cdots\oplus K_2}_{\rm elements \
of \ order \ 3},$$ consequently it contains no cycles and so
$\mathcal{P}^\ast (G)$ is bipartite, as desired.
\end{proof}

An immediate consequence of Theorem \ref{2partite} is the
following.

\begin{corollary}
Let $G$ be a nontrivial finite group. Then $\mathcal{P}^\ast (G)$
is a tree if and only if $G=\mathbb{Z}_2$ or $\mathbb{Z}_3$.
\end{corollary}

A graph is {\em planar} if it has a drawing in the plane without
crossing edges. In the following result we characterize the
groups $G$ for which the graph $\mathcal{P}^\ast (G)$ is planar.
\begin{theorem}\label{planar}
Let $G$ be a nontrivial finite group. Then the following
conditions are equivalent.

$(a)$ $\mathcal{P}^\ast (G)$ is planar.

$(b)$ $\pi_e(G)\subseteq \{1, 2, 3, 4, 5, 6\}$.
\end{theorem}
\begin{proof} $(a)\Longrightarrow (b)$. Let $x$ be an arbitrary element
of $G$ and set $H:=\langle x\rangle$. Then, by Lemma
\ref{elementary-results} $(a)$, $\mathcal{P}^\ast (H)$ is a
subgraph of $\mathcal{P}^\ast (G)$. On the one hand, if $x$ has
prime-power order and $o(x)\geq 7$, then $\mathcal{P}^\ast (H)$
is complete, thus $\mathcal{P}^\ast (H)$ and so $\mathcal{P}^\ast
(G)$ contains $K_5$, which shows that $\mathcal{P}^\ast (G)$ is
not planar (Theorem 6.2.2, \cite{west}). With the similar
argument, we can verify that if $\phi(o(x))\geq 5$, then $H$ has
at least $5$ generators, which forces $\mathcal{P}^\ast (H)$ and
so $\mathcal{P}^\ast (G)$ again contains $K_5$, and hence
$\mathcal{P}^\ast (G)$ is not planar. In particular, we conclude
that $\pi(G)\subseteq \{2, 3, 5\}$ and $\pi_e(G)\subseteq \{1, 2,
3, 4, 5, 6, 10, 12\}$.

Furthermore, if $G$ contains an element of order $10$ or $12$, say
$x$, then the cyclic subgroup $\langle x\rangle$ has four
generators, each of them is adjacent to all other elements of
$\langle x\rangle$ in $\mathcal{P}^\ast(\langle x\rangle)$. Now
the induced subgraph on these generators and the vertex $x^2$ is
a clique which is isomorphic to $K_5$. Thus $\mathcal{P}^\ast
(\langle x\rangle)$ and so $\mathcal{P}^\ast (G)$ contains $K_5$,
which is not planar as before. Therefore $\pi_e(G)\subseteq \{1,
2, 3, 4, 5, 6\}$, as required.

$(b)\Longrightarrow (a)$. In the case when $\pi_e(G)\subseteq \{1,
2, 3, 4, 5\}$, $G$ is an $EPPO$-group and from Lemma
\ref{elementary-results} $(b)$ it follows that the power graph
$\mathcal{P}(G)$ is the union of complete subgraphs $K_2$, $K_3$,
$K_4$ and $K_5$ with exactly one common vertex, i.e., the identity
element. Hence the graph $\mathcal{P}^\ast(G)$ is a disjoint union
of complete subgraphs $K_1$, $K_2$, $K_3$ and $K_4$, and so it
does not contain a subdivision of $K_5$ or $K_{3,3}$, which shows
that $\mathcal{P}^\ast (G)$ is planar (Theorem 6.2.2,
\cite{west}).

Now, we may assume that $6\in \pi_e(G)$. In this case the only
possibilities for $\pi_e(G)$ are:
$$ \{1, 2, 3, 6\}, \ \ \{1, 2, 3, 4, 6\}, \ \ \{1, 2, 3, 5, 6\} \ \ {\rm and} \ \  \{1, 2, 3, 4, 5, 6\},$$
or equivalently $6\in \mu(G)\subseteq \{4, 5, 6\}$. Note that, if
$5\in \pi(G)$, then $G$ is $C_{5,5}$-group (i.e., a group whose
order is divisible by $5$ and in which the centralizer of a
$5$-element is a $5$-group). Thus for every $5$-element of $G$,
say $x$, the subgraph $\mathcal{P}^\ast(\langle x\rangle)\cong
K_4$ is a connected component of $\mathcal{P}^\ast(G)$ (see Lemmas
\ref{complete-CSS} and \ref{elementary-results} $(b)$). In the
case when $G$ contains an element of order $4$, say $y$, noting
$4\in \mu(G)$ a similar reasoning shows that the subgraph
$\mathcal{P}^\ast(\langle y\rangle)\cong K_3$ is a connected
component of $\mathcal{P}^\ast(G)$. In general case we have
$$\mathcal{P}^\ast(\mathbb{Z}_n)\cong K_{n-1}, \ \ \ n=2, 3, 4, 5,$$
each of them is planar. The power graph
$\mathcal{P}^\ast(\mathbb{Z}_6)$ is also planar, indeed we have
the following planar drawing of the power graph
$\mathcal{P}^\ast(\mathbb{Z}_6)$:

\vspace{1.5cm} \setlength{\unitlength}{4mm}
\begin{picture}(1,9)(5,1)
\linethickness{0.3pt} %
\put(21,8){\circle*{0.5}}%
\put(27,8){\circle*{0.5}}%
\put(24,12.5){\circle*{0.5}}%
\put(24,9.5){\circle*{0.5}}%
\put(21,8){\line(1,0){6}}%
\put(21,8){\line(2,3){3}}%
\put(24,12.5){\line(2,-3){3}}%
\put(24,9.5){\line(0,1){3}}%
\put(21,8){\line(2,1){3}}%
\put(24,9.5){\line(2,-1){3}}%
\put(21,8){\line(2,-1){3}}%
\put(27,8){\line(-2,-1){3}}%
\put(24,6.5){\circle*{0.5}}%
\put(17,4.5){{\small \bf Figure 4.} \ The power graph \small$\mathcal{P}^\ast(\mathbb{Z}_6)$}%
\end{picture}

\noindent As discussed above, we have
$$\mathcal{P}^\ast(G)=\bigcup_{x\in G}\mathcal{P}^\ast(\langle x\rangle)=K_4\oplus K_4\oplus \cdots \oplus K_4\oplus
\bigcup_{o(x)\neq 5}\mathcal{P}^\ast(\langle x\rangle),$$ On the
other hand, since $$|\mathbb{Z}_2\cap\mathbb{Z}_4|\leq 2, \ \
|\mathbb{Z}_2\cap\mathbb{Z}_6|\leq 2, \ \
|\mathbb{Z}_4\cap\mathbb{Z}_6|\leq 2 \ \ \ \mbox{and} \ \ \
|\mathbb{Z}_3\cap\mathbb{Z}_6|\leq 3,$$ the subgraph
$$\bigcup_{o(x)\neq 5}\mathcal{P}^\ast(\langle x\rangle)$$ of
$\mathcal{P}^\ast(G)$ consist of some planar graphs sharing a
common vertex or a common edge, which implies that
$\bigcup_{o(x)\neq 5}\mathcal{P}^\ast(\langle x\rangle)$ and so
$\mathcal{P}^\ast(G)$ is planar. This completes the proof.
\end{proof}

In \cite{Huang-Shi, shi-wenze, shi-yang-86, Shi-Yang}, under the
assumption of finiteness of a group the authors study the
structure of $EPPO$, $EPO$, $EPPOM$ and $EPOM$-groups. In
\cite{Gupta-Mazurov, Levi-Wrden, Lytkina, Mazurov-60}, the
authors determined also groups with small orders of elements.

{\bf Remark 2.}  Let $G$ be a non-trivial finite group with
$\pi_e(G)\subseteq \{1, 2, 3, 4, 5, 6\}$. The classification of
all such groups are listed below:

\begin{tabular}{ll}
$(1)$  & $\pi_e(G)=\{1, 2\}$ and $G$ is an elementary abelian
$2$-group.\\
$(2)$ & $\pi_e(G)=\{1, 3\}$ and $G$ is nilpotent of class at most
3 (\cite{Levi-Wrden}).\\
$(3)$ & $\pi_e(G)=\{1, 5\}$ and $G$ is a $5$-group of exponent 5.\\
$(4)$ &$\pi_e(G)=\{1, 2, 3\}$ and $G=[N]K$ is a Frobenius group
where either $N\cong \mathbb{Z}_3^t$, $K\cong \mathbb{Z}_2$ \\ &
or $N\cong
\mathbb{Z}_2^{2t}$, $K\cong \mathbb{Z}_3$ (\cite{Neumann}).\\
$(5)$ &  $\pi_e(G)=\{1, 2, 4\}$ and $G$ is a $2$-group of exponent $4$. \\
$(6)$ &  $\pi_e(G)=\{1, 2, 5\}$ and $G$ is a $EPO$-group. (\cite{shi-wenze})\\
$(7)$ & $\pi_e(G)=\{1, 3, 5\}$ and $G=[N]K$ is a Frobenius group
where either\\ & $\bullet$ $N$ is a $5$-group which is nilpotent
of class at most $2$ and $|K|=3$, or \\ & $\bullet$ $N$ is a
$3$-group which is nilpotent of class at most $3$ and $K$ is a
$5$-group. (\cite{Gupta-Mazurov})\\
$(8)$ & $\pi_e(G)=\{1, 2, 3, 4\}$ and $G=[N]K$ and one of the following occurs: \cite{Brandl-Shi, Lytkina} \\
& $\bullet$  $N=\mathbb{Z}_3^{2t}$ and $K\cong \mathbb{Z}_4$ or $K\cong Q_8$, and $G$ is a Frobenius group. \\
& $\bullet$  $N=\mathbb{Z}_2^{2t}$ and $K\cong S_3$. \\
& $\bullet$  $N$ is a $2$-group with exponent $4$ and of class $\leq 2$ and $K\cong \mathbb{Z}_3$. \\
$(9)$ & $\pi_e(G)=\{1, 2, 4, 5\}$ and $G=[N]K$ and one of the following occurs: (\cite{Gupta-Mazurov}) \\
& $\bullet$ $N$ is an elementary abelian $2$-group and $K$ is a
non-abelian group of order 10. \\ & $\bullet$ $N$ is an
elementary abelian $5$-group and $K$ is isomorphic to a subgroup
of $Q_8$.\\ & $\bullet$ $N$ is a $2$-group
which is nilpotent of class at most $6$ and $K$ is a $5$-group.\\
$(10)$ & $\pi_e(G)=\{1, 2, 3, 5\}$ and $G\cong A_5$. (\cite{shi-wenze, Z-Maz})\\
$(11)$ & $\pi_e(G)=\{1, 2, 3, 6\}$ and $G=PQ$ is a $\{2,
3\}$-group of exponent $6$, where $P$ is an \\ &  elementary
abelian $2$-group and
$Q$ is a $3$-group with exponent $3$.\\ &  (\cite{Shi-Yang}, Theorem 1 $(III)$)\\
$(12)$ & $\pi_e(G)=\{1, 2, 3, 4, 5\}$ and one of the following
holds: (\cite{Brandl-Shi, Mazurov-60})\\
& $\bullet$ $G\cong A_6$.\\
& $\bullet$ $G=[N]K$, where $N$ is an elementary
abelian $2$-group and a direct sum of natural\\ & ${\rm SL}(2,4)$-modules, and $K\cong A_5$.\\
$(13)$ & $\pi_e(G)=\{1, 2, 3, 4, 6\}$ and $G$ is a solvable
$EPPMO$-group. In particular, $G$ is a \\
&  $\{2,3\}$-group of exponent $12$. (\cite{Huang-Shi}, Theorem 2.1 $(I)$)\\
$(14)$ & $\pi_e(G)=\{1, 2, 3, 5, 6\}$ and $G$ is a solvable group.
(\cite{Yang-Wang}, Theorem 2).\\
$(15)$ & $\pi_e(G)=\{1, 2, 3, 4, 5, 6\}$ and one of the following
holds: (\cite{Brandl-Shi})\\
& $\bullet$ $G=[\mathbb{Z}_5^{2t}]K$ is a Frobenius group, where
$K\cong [\mathbb{Z}_3]\mathbb{Z}_4$
or $K\cong {\rm SL}(2,3)$.\\
& $\bullet$ $G/O_2(G)\cong A_5$ and $O_2(G)$ is elementary
abelian and a direct sum of natural and\\ & orthogonal ${\rm SL}(2,4)$-modules.\\
& $\bullet$ $G=S_5$ or $G=S_6$.
\end{tabular}
\section{The Power Graph of Finite Group Which is Eulerian}
Recall that a trail in a graph $\Gamma$ is a walk with no
repeated edge. A trail that traverses every edge and every vertex
of $\Gamma$ is called an Eulerian trail. A closed Eulerian trail
is called an Eulerian circuit. A connected graph is said to be
Eulerian if it contains an Eulerian circuit, and non-Eulerian
otherwise. A well-known theorem due to Euler states {\em a
connected graph $\Gamma$ is Eulerian if and only if all the
vertices of $\Gamma$ are of even degree}. We would like to
examine now the cyclic groups $G$ for which the graph
$\mathcal{P}^\ast (G)$ is Eulerian. Our principal result in this
section is the following.
\begin{proposition}\label{proposition4} Let $G$ be a cyclic group of order $n$. Then $\mathcal{P}^\ast (G)$ is Eulerian
if and only if $n$ is a power of $2$.
\end{proposition}
\begin{proof} Suppose first that $n$ is a power of $2$. Then it follows from Lemma \ref{l1} that $\mathcal{P}^\ast
(G)$ is a complete graph, and so the degree of all vertices of
$\mathcal{P}^\ast (G)$ is $n-2$, which is even. This shows that
$\mathcal{P}^\ast (G)$ is Eulerian.

Conversely, we assume that $\mathcal{P}^\ast (G)$ is Eulerian. We
want to prove that $n$ is a power of $2$. First of all, if $x$ is
a generator of $G$, then by Lemma \ref{elementary-results} $(b)$,
$d(x)=o(x)-2=n-2$, which shows that $n$ must be even. We claim now
that $n$ is a power of $2$. Assume the contrary. Let $n$ have
prime-power factorization
$$n=p_1^{\alpha_1}p_2^{\alpha_2}\cdots p_k^{\alpha_k},$$
where $p_1=2, p_2, \ldots, p_k$ are distinct primes and $k\geq 2,
\alpha_1, \alpha_2, \ldots, \alpha_k$ are positive integers. Then
$$G\cong \mathbb{Z}_{p_1^{\alpha_1}}\times
\mathbb{Z}_{p_2^{\alpha_2}}\times \cdots \times
\mathbb{Z}_{p_k^{\alpha_k}}.$$ Let
$\mathbb{Z}_{p_i^{\alpha_i}}=\langle \bar{x}_i\rangle$, $i=1, 2,
\ldots, k$ and $\bar{x}=(1, \bar{x}_2, \ldots, \bar{x}_k)$.
Clearly $o(\bar{x})=p_2^{\alpha_2}\cdots p_k^{\alpha_k}$. In the
sequel, we will show that
\begin{equation}\label{eee33}
d(\bar{x})=\prod_{i=2}^{k}p_i^{\alpha_i}-2+(2^{\alpha_1}-1)\prod_{i=2}^{k}\phi(p_i^{\alpha_i}),
\end{equation}
which is an odd number, because $\phi(p_i^{\alpha_i})$ is even
for all $i\geq 2$. Therefore in this case we arrive to a
contradiction, since the degree of all vertices in
$\mathcal{P}^\ast (G)$ must be even.

The proof of the claim (\ref{eee33}) requires some calculations.
Let $\bar{z}=(\bar{x}_1^{m_1}, \bar{x}_2^{m_2}, \ldots,
\bar{x}_k^{m_k})$ be an arbitrary element in $G\setminus \{1\}$,
where $0\leq m_i<p_i^{\alpha_i}$, $i=1,\ldots, k$, and
$m_1+m_2+\cdots+m_k\neq 0$. It will be convenient to consider
among three cases:

$(a)$ $m_1\neq 0$ and $m_2=m_3=\cdots=m_k=0$,

$(b)$ $m_1=0$ and $m_2+m_3+\cdots+m_k\neq 0$,

$(c)$ $m_1\neq 0$ and $m_2+m_3+\cdots+m_k\neq 0$.

Suppose first that $(a)$ holds, that is $\bar{z}=(\bar{x}_1^{m_1},
1, \ldots, 1)$, where $1\leq m_1<p_1^{\alpha_1}$. Then
$o(\bar{z})$ is a power of $2$, and since $(o(\bar{x}),
o(\bar{z}))=1$, $\bar{x}\nsim \bar{z}$ in $\mathcal{P}^\ast (G)$
by Lemma \ref{elementary-results} $(c)$.

Suppose next that $(b)$ holds, that is $\bar{z}=(1,
\bar{x}_2^{m_2}, \ldots, \bar{x}_k^{m_k})$, where $0\leq
m_i<p_i^{\alpha_i}$, $i=2,\ldots, k$, and $m_2+\cdots+m_k\neq 0$.
In what follows, we will prove that $\bar{x}\sim \bar{z}$. To do
this, we show that $\bar{z}\in \langle \bar{x} \rangle$, or
equivalently $\bar{z}=\bar{x}^m$ for some $m$. Let us consider
the following system of simultaneous congruences
\begin{equation}\label{e333333}
\left\{\begin{array}{l} x\equiv m_2\pmod{p_2^{\alpha_2}}\\[0.1cm]
x\equiv m_3\pmod{p_3^{\alpha_3}}\\
\ \ \ \vdots \\[0.1cm]
x\equiv m_k\pmod{p_k^{\alpha_k}}.\\
\end{array} \right.
\end{equation}
Since $p_2^{\alpha_2}, p_3^{\alpha_3}, \ldots, p_k^{\alpha_k}$
are pairwise relatively prime, the Chinese Remainder Theorem
tells us that there is a solution for the system of congruences
(\ref{e333333}), say  $x=m$. Now,  it is easy to see that
$\bar{z}=\bar{x}^m$, as required.

Finally, assume that $(c)$ holds, that is
$\bar{z}=(\bar{x}_1^{m_1}, \bar{x}_2^{m_2}, \ldots,
\bar{x}_k^{m_k})$, where $0\leq m_i<p_i^{\alpha_i}$, $m_1\neq 0$
and $m_2+\cdots+m_k\neq 0$. In this case, we claim that
\begin{center} {\em
$\bar{x}\sim \bar{z}$  \ if and only if \ $(m_i,
p_i^{\alpha_i})=1$ for each $i=2, 3, \ldots, k$.}
\end{center}

For the proof of the claim, we assume first that $(m_i,
p_i^{\alpha_i})=1$ for each $i=2, 3, \ldots, k$.  Let
$o(\bar{x}_1^{m_1})=l$ which is a power of 2, and let $n_i:=lm_i$
for each $i=2, 3,\ldots, k$. Then we obtain
$$\bar{z}^l=(1, \bar{x}_2^{n_2},
\ldots, \bar{x}_k^{n_k}).$$ Since for each $i=2, 3, \ldots, k$,
we have $(n_i, p_i^{\alpha_i})=1$, $n_i$ has a unique
multiplicative inverse modulo $p_i^{\alpha_i}$, which is denoted
by $n_i^\ast$, that is $n_in_i^\ast\equiv 1
\pmod{p_i^{\alpha_i}}$. Now, we consider the system of
simultaneous congruences
\begin{equation}\label{e4434}
\left\{\begin{array}{l} x\equiv n_2^\ast\pmod{p_2^{\alpha_2}}\\[0.1cm]
x\equiv n_3^\ast\pmod{p_3^{\alpha_3}}\\
\ \ \ \vdots \\[0.1cm]
x\equiv n_k^\ast\pmod{p_k^{\alpha_k}}.\\
\end{array} \right.
\end{equation}
Again, by the Chinese Remainder Theorem there exists a solution
for the system of congruences (\ref{e4434}), say $x=n$.
Therefore, we obtain $$\bar{z}^{ln}=(1, \bar{x}_2^{n_2n}, \ldots,
\bar{x}_k^{n_kn})=(1, \bar{x}_2, \ldots, \bar{x}_k)=\bar{x},$$
which shows that $\bar{x}\in \langle \bar{z}\rangle$, and so
$\bar{x}\sim \bar{z}$, as claimed.

Now suppose that there exists an integer $i\in \{2, 3, \ldots,
k\}$ such that $(m_i, p_i^{\alpha_i})\neq 1$. Hence $p_i$ divides
$m_i$. We want to show that $\bar{x}\nsim \bar{z}$. Suppose to
the contrary that $\bar{x}\sim \bar{z}$. Then by the definition
we have $\bar{x}\in \langle\bar{z} \rangle$ or $\bar{z}\in
\langle\bar{x} \rangle$. Clearly $\bar{z}\notin \langle\bar{x}
\rangle$, because $\bar{x}_1^{m_1}\neq 1$. On the other hand, if
$\bar{x}\in \langle\bar{z} \rangle$, then $\bar{x}=\bar{z}^l$ for
some $l$. In particular, we conclude that
$\bar{x}_i=\bar{x}_i^{lm_i}$, and this means that $lm_i\equiv
1\pmod{p_i^{\alpha_i}}$, or equivalently $p_i^{\alpha_i}$ and so
$p_i$ divides $lm_i-1$. But this is contrary to the fact that
$p_i|m_i$.

Finally, we are now ready to calculate the degree of vertex
$\bar{x}$ in $\mathcal{P}^\ast (G)$. By what observed above
(cases $(a)-(c)$) we obtain
$$\begin{array}{lll}
d(\bar{x})&=& |\langle \bar{x}\rangle\setminus \{1,
\bar{x}\}|+|\{(\bar{x}_1^{m_1}, \bar{x}_2^{m_2}, \ldots,
\bar{x}_k^{m_k}) \ | \ 0<m_1<p_1^{\alpha_1}, (m_i,
p_i^{\alpha_i})=1, 2\leq i\leq k \}|\\[0.3cm]
&= &
\prod_{i=2}^{k}p_i^{\alpha_i}-2+(2^{\alpha_1}-1)\prod_{i=2}^{k}\phi(p_i^{\alpha_i}),\\
\end{array}$$
which completes the proof.
\end{proof}
\section{Cut-edges in Power Graphs}
For any edge $e$ of a graph $\Gamma$, if
$c(\Gamma-e)=c(\Gamma)+1$, the edge $e$ is called a cut edge of
$\Gamma$, where $c(\Gamma)$ denotes the number of connected
components of $\Gamma$. Note that the following characterization
of cut edges is well-known: {\em An edge $e$ of a graph $\Gamma$
is a cut edge if and only if $e$ belongs to no cycle of $\Gamma$.}
\begin{theorem}\label{cutedges} Let $G$ be a finite group and $\Gamma=\mathcal{P}^\ast
(G)$. An edge $e=xy\in \Gamma$ is a cut edge if and only if
$d_{\Gamma}^-(x)=d_{\Gamma}^+(x)=1$.
\end{theorem}
\begin{proof} ($\Longrightarrow$) Assume that $e=xy$ is a cut
edge of $\Gamma$. Then by definition we have $y\in \langle
x\rangle$ or $x\in \langle y\rangle$, and hence $o(y)|o(x)$ or
$o(x)|o(y)$. Without loss of generality we may assume that $y\in
\langle x\rangle$ and so $o(y)|o(x)$. If $o(x)\geq 4$, then
$\phi(o(x))\geq 2$, this means that the cyclic group $\langle x
\rangle$ has at least two generators which forces $e$ lies in a
cycle of $\Gamma$. Therefore $o(x)\leq 3$. If $o(x)=2$, then $x$
and $y$ would be two involutions which is joined be an edge, a
contradiction by Lemma \ref{elementary-results} $(c)$.  Finally,
we conclude that $o(x)=3$, $y=x^2$ and from Lemma
\ref{elementary-results-digraphs} $(e)$ it follows that
$N_{\Gamma}^-(x)\setminus\{x^2\}=N_{\Gamma}^-(x^2)\setminus\{x\}(=\emptyset$,
otherwise $e=xx^2$ lies on a cycle of $\Gamma$) and
$N_{\Gamma}^+(x)\setminus\{x^2\}=N_{\Gamma}^+(x^2)\setminus\{x\}(=\emptyset$,
because $x$ an element of order 3). Therefore
$d_{\Gamma}^-(x)=d_{\Gamma}^+(x)=1$.

($\Longleftarrow$) Conversely, assume that
$d_{\Gamma}^-(x)=d_{\Gamma}^+(x)=1$. Then $N_\Gamma^+(x)=\{1,
x^2\}$, $N_\Gamma^-(x)=\{x^2\}$ and so $o(x)=3$. This means that
$y=x^2$ and $e=xx^2$ is a singleton edge in $\Gamma$, and hence
the vertices $x$ and $x^2$ are two isolated vertices in
$\Gamma-e$, which implies that $c(\Gamma-e)=c(\Gamma)+1$.
Therefore $e$ is a cut edge in $\Gamma$, as claimed.
\end{proof}

The condition $d_{\Gamma}^-(x)=d_{\Gamma}^+(x)=1$ in Theorem
\ref{cutedges} is equivalent to the following condition:

{\em  ``The group $G$ has an element $x$ of order $3$ such that
for all elements $y\in G\setminus \langle x\rangle$,
$x\not\in\langle y\rangle$."}

There are many examples of such groups, for instance:
$$S_3=[\mathbb{Z}_3]\mathbb{Z}_2, \ \  F_{21}=[\mathbb{Z}_7]\mathbb{Z}_3, \ \  S_4, \ \ S_3\times \mathbb{Z}_3  \ \ {\rm
and} \ \ S_3\times S_3.$$ Moreover, an infinite family of such
group is
$$[\mathbb{Z}_7]\mathbb{Z}_3, \ \ [\mathbb{Z}_7\times \mathbb{Z}_7]\mathbb{Z}_3,  \ \ [\mathbb{Z}_7\times
\mathbb{Z}_7\times \mathbb{Z}_7]\mathbb{Z}_3, \ \
[\mathbb{Z}_7\times \mathbb{Z}_7\times \mathbb{Z}_7\times
\mathbb{Z}_7]\mathbb{Z}_3, \ldots$$ It is worth noting that this
family of groups can be regarded as a special case of
$EPPO$-groups. As a matter of fact, since the power graph
associated with an $EPPO$-group is the union of complete subgraphs
with exactly one common vertex (Lemma \ref{elementary-results}
$(b)$), all $EPPO$-groups $G$ with spectrum $\pi_e(G)=\{1, 3, p^s,
q^t, \ldots\}$ are such examples. Therefore, we focus our
attention on $EPPO$-groups. The complete classification of finite
$EPPO$-groups is given in \cite[Theorems 2.4 and
3.1]{shi-yang-86}.
\begin{theorem}\label{theorem-Shi} Let $G$ be a finite
$EPPO$-group. Then we have
\begin{itemize}
\item[$(1)$] If $G$ is solvable, then $|\pi(G)|\leq 2$. Moreover, if $|G|=p^\alpha
q^\beta$, $P_1$ is the maximal normal $p$-subgroup of $G$ and
$|P_1|=p^\gamma$, then the Sylow subgroups of $G/P_1$ are cyclic
or generalized quaternion, $p^{\alpha-\gamma}\mid q-1$ and
$\gamma=kb$ where $b$ is the exponent of $p$ $\pmod{q^\beta}$ for
the case of cyclic or $p$ (mod $2^{\beta-1}$) for the case of
generalized quaternion ( in this case $\gamma = \alpha$).
\item[$(2)$] If $G$ is non-solvable, then one of the following hold:

$\bullet$ $G$ is simple and $G\cong L_2(q)$, $q=5, 7, 8, 9 , 17$;
$L_3 (4)$, ${\rm Sz}(8)$ or ${\rm Sz}(32)$,

$\bullet$  $G\cong M_{10}$, or

$\bullet$  $G$ has an elementary abelian $2$-subgroup $P$, $P\lhd
G$ and $\frac{G}{P}\cong L_2(5)$, $L_2(8)$, ${\rm Sz}(8)$ or ${\rm
Sz}(32)$. \end{itemize}
\end{theorem}

An immediate consequence of Theorem \ref{theorem-Shi} is the
following which gives the structure of finite $EPPO$-groups $G$
with spectrum $\pi_e(G)=\{1, 3, p^s, q^t, \ldots\}$.
\begin{corollary}\label{eppo-group}
Let $G$ be an $EPPO$-group with $\pi_e(G)=\{1, 3, p^s, q^t,
\ldots\}$, where $p, q, \ldots$ are primes not equal to $3$. Then
$G$ is isomorphic to one of the following groups:
\begin{itemize}
\item[$(1)$]  If $G$ is solvable and non-nilpotent, then $|G|=3^\alpha
q^\beta$. If $Q$ is the maximal normal $q$-subgroup ($q \neq 3$)
of $G$, then

$(1.1)$ $q\neq 2$, $G$ has the chief factors $3, 3, \ldots, 3;
q^b, q^b, \ldots, q^b$, $\beta=kb$ and $b$ is the exponent of $q$
$\pmod {3^\alpha}$, $Q$ is the Sylow $q$-subgroup of $G$

$(1.2)$ $q=2$ and the Sylow $2$-subgroups of $G$ are not
generalized quaternion, $G$ has the chief factors $2; 3, 3,
\ldots, 3; 2^{b_1}, 2^{b_2}, \ldots,  2^{b_k}$, $b|b_i$ and $b$ is
the exponent of $3$ $\pmod{2^\beta}$.

If the maximal normal subgroup of $G$ is a $3$-subgroup, then

$(1.3)$ $q=2$ and the Sylow $2$-subgroups of $G$ are generalized
quaternion, $G$ has a chief factors $2, 2, \ldots, 2; 3^{b_1},
3^{b_2}, \ldots, 3^{b_k}$, $b|b_i$ and $b$ is the exponent of $3$
$\pmod{2^{\beta-1}}$.

\item[$(2)$] If $G$ is not solvable, then $G$ is one of the following groups:

$(2.1)$  $G$ is simple and $G\cong L_2(q)$, $q=5, 7, 9$; or
$L_3(4)$.

$(2.2)$ $G\cong M_{10}$, or

$(2.3)$ $G$ has an elementary abelian $2$-subgroup $Q$, $Q\unlhd
G$ and $G/Q\cong L_2(5)$.
\end{itemize}
\end{corollary}

\section{On the Number of Edges in $\mathcal{P}^\ast (G)$ and Related Results}

\begin{lm}\label{lemma10} Let $G$ be a finite group. Then the number
of edges $e^\ast$ of $\mathcal{P}^\ast (G)$ is given by
$$2e^\ast=\sum_{g\in G^\#}\big(2o(g)-\phi(o(g))-3\big),$$
where $G^\#=G\setminus \{1\}$. Especially, if $G$ is an elementary
$p$-group of order $p^m$, then $$e^\ast=(p^m-1)(p-2)/2.$$ In
particular, if $G$ is an elementary abelian $2$-group,  then the
graph $\mathcal{P}^\ast (G)$ is a null graph.
\end{lm}
\begin{proof} The first assertion follows immediately from \cite[Theorem
4.2]{CSS}. The rest of lemma can be verified by direct
computations.
\end{proof}

An immediate consequence of Lemma \ref{lemma10} is the following.

\begin{corollary}\label{corollary3} Let $G$ be a finite group. Then, there holds
$$2e^\ast=\sum_{n\in \pi_e(G)\setminus\{1\}}s_n\big(2n-\phi(n)-3\big),$$
where $s_n$ is the number of elements with order $n$.
\end{corollary}

The next result is also a simple consequence of Theorem
\ref{planar} and Lemma \ref{lemma10}.
\begin{corollary}
Let $G$ be a nontrivial finite group such that the graph
$\mathcal{P}^\ast (G)$ is $2$-partite. Then the number of edges
of $\mathcal{P}^\ast (G)$ is equal to half the number of elements
of order $3$ in $G$.
\end{corollary}

Below is a number-theoretic proposition, whose validity is
verified by direct computations. However, we will prove it using a
graph approach.

\begin{proposition}\label{proposition4} Let $p$ be a prime number and $n$ a positive integer.
Then there holds
$$\sum_{i=1}^n\phi(p^i)\big(2p^i-\phi(p^i)-3\big)=2{p^n-1\choose 2}.$$
\end{proposition}
\begin{proof} Let $G$ be a cyclic group of order $p^n$. Then, on the one
hand, the graph $\mathcal{P}^\ast (G)$ is a complete graph on
$p^n-1$ vertices and then the number of its edges is equal to
\begin{equation}\label{equation3}
e^\ast={p^n-1\choose 2}.
\end{equation}
On the other hand, we know that $\pi_e(G)=\{1, p, p^2, \ldots,
p^n\}$ and by Corollary \ref{corollary3}, we obtain that
\begin{equation}\label{equation4}
2e^\ast=\sum_{i=1}^ns_{p^i}\big(2p^i-\phi(p^i)-3\big),
\end{equation}
where $s_{m}$ signifies the number of elements with order $m$.
Note that $s_m=k\phi(m)$, where $k$ is number of cyclic subgroups
of order $m$ and $\phi(m)$ Euler totient function. Moreover, since
$G$ is a cyclic group, $k=1$ for all $m\in \pi_e(G)$. Thus,
$s_{p^i}=\phi(p^i)$ for each $i=1, 2, \ldots, n$. If this is
substituted in Eq. (\ref{equation4}), then we obtain
\begin{equation}\label{equation5}
e^\ast=\frac{1}{2}\sum_{i=1}^n\phi(p^i)\big(2p^i-\phi(p^i)-3\big).
\end{equation}
The result now follows by comparing Eqs. (\ref{equation3}) and
(\ref{equation5}). \end{proof}


\begin{thebibliography}{99}
\bibitem{AAM} A. Abdollahi, S. Akbari  and H. R. Maimani, {\em Non-commuting
graph of a group}, J. Algebra, 298(2)(2006), 468-492.

\bibitem{Ba} C. Bates, D. Bundy, S. Perkins and P. Rowley, {\em  Commuting
involution graphs for symmetric groups}, J. Algebra, 266 (2003),
133-153.

\bibitem{Bi} J. X. Bi, {\em A characterization of symmetric
groups}, Acta Math. Sinica, 33(1990), 70-77. (Chinese)


\bibitem{Brandl-Shi} R. Brandl and W. J. Shi, {\em Finite groups whose element orders are
consecutive integers}, J. Algebra, 143(2)(1991), 388-400.

\bibitem{C1} P. J. Cameron and S. Ghoshb, {\it The power graph of a finite
group}, Discrete Math., 311(13)(2011), 1220-1222.

\bibitem{C2}  P. J. Cameron, {\it The power graph of a finite
group. II},  J. Group Theory,  13(6)(2010), 779-783.

\bibitem{Cao} H. P. Cao and W. J. Shi, {\em Pure quantitative characterization of
finite projective special unitary groups}, Sci. China, Ser. A,
45(6)(2002), 761-772.

\bibitem{CSS} I. Chakrabarty, S. Ghosh and M. K. Sen, {\em Undirected power graphs
of semigroups}, Semigroup Forum, 78 (2009), 410-426.

\bibitem{Ch} G. Chartland and L. Lesniak, {\em Graphs and Digraphs}, Chapman
$\&$ Hall, London, 1996.

\bibitem{Cli} A. H. Clifford and G. B. Preston, {\em The Algebraic Theory of
Semigroups}, Amer. Math. Soc., Providence, 1961.

\bibitem{Godsil-Royle}  C. Godsil and G. Royle, {\em Algebraic Graph Theory},  Springer-
Verlag, New York, 2001.

\bibitem{Gor} D. Gorenstein, {\em Finite Groups}, (Second Edition), Chelsea Publishing Co., New York, 1980.

\bibitem{Gr} P. A. Grillet, {\em Semigroups: An Introduction to Structure
Theory}, Dekker, New York, 1995.

\bibitem{Gru} K. W. Gruenberg, {\em Free abelianised extensions of finite groups}, In
C. T. C. Wall, editor, Homological group theory. Proc. Symp.,
Durham 1977, volume 36 of Lond. Math. Soc. Lect. Note Ser.,
71-104, Cambridge, 1979. London Mathematical Society, Cambridge
Univ. Press.

\bibitem{Gupta-Mazurov} N. D. Gupta and V. D. Mazurov,{\em
On groups with small orders of elements}, Bull. Austral. Math.
Soc. 60(2)(1999), 197-205.


\bibitem{Ho} J. M. Howie, {\em Fundamentals of Semigroup Theory}, Clarendon Press,
Oxford, 1995.

\bibitem{Huang-Shi} T. L. Huang and W. J. Shi, {\em Finite groups all of whose element
orders are of prime power except one}, J. Southwest Normal Univ.,
20(6)(1995), 610-617. (in Chinese).

\bibitem {IY1} N. Iiyori and H.  Yamaki, {\em  Prime graph
components of the simple groups of Lie type over the field of
even characteristic}, Proc. Japan Acad. Ser. A Math. Sci., 67 (3)
(1991), 82-83.

\bibitem {IY2} N. Iiyori and H.  Yamaki,
{\em Prime graph components of the simple groups of Lie type over
the field of even characteristic}, J. Algebra 155 (2)(1993),
335-343. Corrigendum: {\em Prime graph components of the simple
groups of Lie type over the field of even characteristic}, J.
Algebra 181 (1996), no. 2, 659.

\bibitem{K1} A. V. Kelarev and  S. J. Quinn, {\em
A combinatorial property and power graphs of groups},
Contributions to General Algebra, 12 (Heyn, Klagenfurt, 2000)
229-235.

\bibitem{K2} A. V. Kelarev, S. J. Quinn and  R. Smolнkovб, {\em Power
graphs and semigroups of matrices}, Bull. Austral. Math. Soc.,
63(2)(2001), 341-344.

\bibitem{K3} A. V. Kelarev and S. J. Quinn, {\em Directed graph and combinatorial
properties of semigroups}, J. Algebra, 251(1)(2002), 16-26.

\bibitem{Kh} A. Khosravi and B. Khosravi, {\em
A new characterization of almost sporadic groups}, J. Algebra
Appl., 1(3)(2002), 267-279.

\bibitem {K}  A. S. Kondrat\'ev, {\em Prime graph components of finite
simple groups},  Math. Sb., 180(6)(1989), 787-797.

\bibitem {Levi-Wrden} F. Levi and B. L. van der Waaerden, {\em
Uber eine besondere Klasse von gruppen}, Abh. Mth. Sem. Univ.
Hamburg, 9 (1932), 154-158.

\bibitem {Lytkina}  D. V. Lytkina, {\em Structure of
a group with elements of order at most $4$}, Siberian Math. J.,
48 (2)(2007), 283-287.

\bibitem {Mazurov-60} V. D. Mazurov, {\em  Groups of exponent $60$ with prescribed orders
of elements}, Algebra and Logic 39(3)(2000), 189-198.

\bibitem{Mszz} A. R. Moghaddamfar, W. J. Shi, W. Zhou and
A. R. Zokayi, {\it On the noncommuting graph associated with a
finite group}, Siberian Mathematical Journal 46(2)(2005), 325-332.

\bibitem{Mogh} A. R. Moghaddamfar, {\em About noncommuting
graphs},  Siberian Mathematical Journal, 47(5)(2006), 911-914.

\bibitem{Neumann} B. H. Neumann, {\em Groups whose elements have bounded orders}, J. London Math.
Soc., 12 (1937), 195-198.

\bibitem{Ne} B. H. Neumann, {\em A problem of Paul Erd$\ddot{o}$s on groups}, J. Austral. Math.
Soc. Ser. A, 21(1976), 467-472.

\bibitem{Rob} D. J. S. Robinson, {\em A Course in the Theory of Groups},
Springer-Verlag, New York/Berlin, 1982.

\bibitem{r} J. S. Rose, {\em A course on group theory}, Cambridge
University Press, 1978.

\bibitem{Se} Y. Segev, {\em On finite homomorphic images of the
multiplicative group of a division algebra}, Ann. of Math., 149
(1999), 219-251.

\bibitem{SS} Y. Segev and G. Seitz, {\em Anisotropic groups of
type $A_n$ and the commuting graph of finite simple groups},
Pacific J. Math., 202(2002), 125-225.

\bibitem{Shen}  H. Shen, H. P. Cao and G. Y. Chen, {\em Characterization of the
automorphism groups of sporadic simple groups}, to appear in
Front. Math. China.

\bibitem{shi-J1} W. J. Shi, {\em A characteristic property of $J_1$ and
${\rm PSL}_2(2^n)$}, Adv. in Math., 16(1987), 397-401. (in
Chinese)

\bibitem{shi-mathieu} W. J. Shi, {\em A characteristic property of Mathieu groups}, Chinese
Ann. Math., 9A(5)(1988), 575-580. (in Chinese)


\bibitem{S1} W. J. Shi, {\em A new characterization of the sporadic simple groups},
in Group Theory, Proc. 1987 Singapore Group Theory Conf., Walter
de Gruyter, New York (1989), 531-540.

\bibitem{S4} W. J. Shi, {\em The pure
quantitative characterization of finite simple groups (I)},
Progr. Nat. Sci., 4(3)(1994), 316-326.

\bibitem{S5} W. J. Shi and J. X. Bi, {\em A characteristic property for each finite
projective special linear group},  Lect. Notes Math., 1456,
Springer-Verlag, Berlin (1990), 171-180.

\bibitem{S3} W. J. Shi and J. X. Bi, {\em A characterization of Suzuki-Ree groups}, Sci.
China, Ser. A, 34(1)(1991), 14-19.

\bibitem{S2} W. J. Shi and J. X. Bi, {\em A new characterization of the alternating
groups}, Southeast Asian Bull. Math., 16(1)(1992), 81-90.

\bibitem{shi-wenze} W. J. Shi and W. Z. Yang, {\em A new characterization of $A_5$ and the finite groups
in which every non-identity element has prime order},  J.
Southwest-China Teachers College, 9(1) (1984), 36-40. (in
Chinese).

\bibitem{shi-yang-86} W. J. Shi and W. Z. Yang, {\em The finite groups all of whose
elements are of prime power order}, J. Yunnan Educational College,
1 (1986), 2-10. (Chinese)

\bibitem{Shi-Yang} W. J. Shi and C. Yang, {\em A class of special finite groups}, Chinese
Science Bull., 37(3)(1992), 252-253.

\bibitem{Un} {\em Unsolved Problems in Group Theory}, The Kourovka Notebook, 16th
edn., Institute of Mathematics SO RAN, Novosibirsk (2006),
http://www.math.nsc.ru/$\sim$alglog.

\bibitem{Vgm} A. V. Vasilev, M. A. Grechkoseeva and
V. D. Mazurov, {\em Characterization of finite simple groups by
spectrum and order}. Algebra and Logic 48(6)(2009), 385-409.

\bibitem{west} D. B. West, {\em Introduction to Graph Theory},
Second Edition, Prentice Hall, Inc., Upper Saddle River, NJ, 2001.

\bibitem{W1} J. S. Williams, {\em The prime graph components of finite groups}, In
B. Cooperstein and G. Mason, editors, The Santa Cruz Conference
on Finite Groups (Proc. Symp. Pure Math., June 25–July 20, 1979,
volume 37 of Proc. Symp. Pure Math., 195-196, Providence, RI, USA,
1980. American Mathematical Society (AMS).

\bibitem{W2} J. S. Williams, {\it Prime
graph components of finite groups}, J. Algebra 69(2)(1981),
487-513.

\bibitem{Xu} M. C. Xu and W. J. Shi, {\em Pure quantitative characterization of finite simple groups
$^2D_n(q)$ and $D_l(q)$ ($l$ odd)}, Alg. Coll., 10(3)(2003),
427-443.

\bibitem{Yang-Wang} C. Yang, S. G. Wang and X. J. Zhao, {\em Finite groups with $\pi_e(G)=
\{1, 2, 3, 5, 6\}$}, Journal of Engineering Mathmatics,
17(1)(2000), 105-108.

\bibitem{Z-Maz} A. Kh. Zhurtov and V. D. Mazurov, {\em On recognition of the
finite simple groups $L_2(2^m)$ in the class of all groups},
Siberian Math. J., 40(1)(1999), 62-64.
\end{thebibliography}
\end{document}